\documentclass[12pt,twoside,reqno]{amsart}

\usepackage{mymacros}
\usepackage{fullpage}
\usepackage[all]{xy}

\usepackage{cmbright}
\usepackage[T1]{fontenc}

\DeclareMathOperator{\Bs}{Bs}
\DeclareMathOperator{\PPic}{\mathbf{Pic}}
\DeclareMathOperator{\AAut}{\mathbf{Aut}}
\newcommand{\bfkbar}{\overline{\bfk}}
\newtheorem*{not_and_conv}{Notation and conventions}

\title[Nonexistence of SOD and canonical bundles]
{Nonexistence of semiorthogonal decompositions
and sections of the canonical bundle}

\author{Kotaro Kawatani and Shinnosuke Okawa}

\date{\today}



\address{Department of Mathematics, Graduate School of Science,
Osaka University, Machikaneyama 1-1, Toyonaka, Osaka 560-0043, Japan}
\email{kawatani@math.sci.osaka-u.ac.jp}
\email{okawa@math.sci.osaka-u.ac.jp}

\begin{document}
 
\begin{abstract}
For any admissible subcategory of the bounded derived category of coherent sheaves on a smooth proper variety, we prove that sections of the canonical bundle impose a strong constraint on the supports of the objects of the subcategory or its semiorthogonal complement. We also show that admissible subcategories are rigid under the actions of topologically trivial autoequivalences.
As applications of these results, we prove that the derived category of various minimal varieties in the sense of the minimal model program admit no non-trivial semiorthogonal decompositions, generalizing the result for curves due to the second author to higher dimensions. The case of minimal surfaces is further investigated in detail.
 \end{abstract}
 
\maketitle
\tableofcontents

%
%

\section{Introduction}
Let $X$ be a smooth proper variety, or more generally a Deligne-Mumford stack
(DM stack for short) defined over a field
$ \bfk $.
The bounded derived category of coherent sheaves on $X$,
which will be denoted by
$
 \bfD^{b} \lb \coh X \rb
$
in this paper, has been acquiring considerable attention from both mathematical and physical points of view.

Sometimes
$
 \bfD^{b} \lb \coh X \rb
$,
as a triangulated $\bfk$-linear category, admits a \emph{semiorthogonal decomposition} (\emph{SOD} for short)
into smaller triangulated subcategories. Since SOD is a fundamental structure of triangulated categories and it often arises as the categorical incarnation of the deep geometry of $X$, it is natural to ask if one can understand the general nature of SODs of
$
 \bfD^{b} \lb \coh X \rb
$.

If the Kodaira dimension of $X$ is non-negative, the most important source of SODs is the \emph{minimal model program} (\emph{MMP} for short).
It is conjectured in general and has been verified in some cases
that each step of MMP induces a non-trivial SOD
of $\bfD ^{ b } \lb \coh X \rb$ (see \cite{MR2483950} and references therein).
Note that the conjecture, in particular, would imply:
\begin{conjecture}\label{cj:MMP_implies_SOD_minimal_case}
If $\bfD ^{ b } \lb \coh X \rb$ does not admit any SOD, then $X$ should be minimal.
\end{conjecture}

When $\dim X = 1$, it is shown in \cite[Theorem 1.1]{MR2838062} that $\bfD ^{ b } \lb \coh X \rb$ admits no SOD if and only if $X$ is minimal in the sense of minimal model program ($\iff g ( X ) \ge 1$). This confirms \pref{cj:MMP_implies_SOD_minimal_case} in dimension 1, and also its converse.
Starting from dimension 2, however, the converse is not correct anymore.

Suppose that $X$ satisfies the following condition.
\begin{equation}\label{eq:acyclicity}
 \bR \Gamma (X, \cO_X) \simeq \bfk
\in
 \bfD ^{ b } ( \Spec \bfk ).
\end{equation}
Then any line bundle $L$ on $X$ is an \emph{exceptional object}
(see, e.g., \cite[Definition 1.57]{MR2244106}), so that it induces a non-trivial SOD
$
 \bfD ^{ b } \lb \coh X \rb = \la \la L\ra^{\perp}, \la L\ra \ra
$.
Here
$
 \la L\ra
$
denotes the smallest triangulated subcategory containing
$
L
$,
and
$
\la L\ra^{\perp}
$
is its right orthogonal complement (\cite[Definition 1.42]{MR2244106}).
Although the condition \pref{eq:acyclicity} is typically satisfied by varieties of negative Kodaira dimension,
such as varieties of Fano type in characteristic zero, there are minimal varieties which also satisfy \pref{eq:acyclicity}.
Already in dimension 2, classical Enriques surfaces and Godeaux surfaces are such examples.
Therefore, even for varieties of non-negative Kodaira dimensions, there are more SODs than MMPs\footnote{Recently several groups worked on exceptional objects
of the derived categories of surfaces of general type satisfying \pref{eq:acyclicity}, and discovered
several examples of (quasi-)phantom categories
(\cite{MR3096524}, \cite{MR3062745}, \cite{MR3361723}, and
\cite{MR3077896}, to name a few).
Another remarkable discovery was that one can also produce a counter-example for
the Jordan-H\"older property of SODs on such a surface
(see \cite{MR3177299}. An easier example was found on a rational threefold in \cite{kuznetsov2013simple}).}.

The purpose of this paper is to carry out the initial step toward understanding/classifying, in arbitrary dimension, the SODs of
$
 \bfD^{b} \lb \coh X \rb
$
for $X$ with non-negative Kodaira dimension. We prove two kinds of constraints which should be fulfilled by \emph{any} SOD of
$
 \bfD^{b} \lb \coh X \rb
$.
As an application we prove the non-existence of non-trivial SODs for many $X$ which are (as expected) minimal.

One of the constraints on SODs
is provided by global/local sections of the canonical bundle $\omega _{ X }$.
In this paper, the base locus of the complete canonical linear system will be denoted by
$ \Bs| \omega_X |$.

\begin{theorem}[{$=$ special case of \pref{th:pg positive}}]
Let $X$ be a smooth proper variety and
$
 \bfD ^{ b } \lb \coh X \rb = \la \cA, \cB \ra
$
an SOD. Then at least one of the followings holds.
	\begin{enumerate}[(a)]
	\item
	the support of any object in $ \cB $
	is contained in
	$ \Bs| \omega_X |$.
	
	\item
	the support of any object in $ \cA $
	is contained in
	$ \Bs| \omega_X |$.
\end{enumerate}
\end{theorem}

This will be used to show the non-existence of SOD for some $X$.
In this paper we always use the following general criterion to show the non-existence of SODs.

\begin{proposition}[{$=$special case of \pref{cr:iff}}]\label{pr:iff_intro}
Let $X$ be a smooth proper variety. Consider an SOD
$
 \bfD ^{ b } \lb \coh X \rb = \la \cA,\cB\ra.
$
Then it is trivial, i.e. $\cA=0$ or $\cB=0$,
if and only if
$
 \cA \otimes \omega_X \subset \cA \subset \bfD ^{ b } \lb \coh X \rb
$
holds.
\end{proposition}

As an immediate corollary, we obtain the following theorem.
\begin{theorem}
Let $X$ be a smooth proper variety such that each connected component of $\Bs| \omega_X |$ is contained in an open subset of $X$ on which $\omega_X$ is trivial. Then $\bfD ^{ b } \lb \coh X \rb$ admits no non-trivial SOD. 
\end{theorem}

A special case of this is

\begin{corollary}\label{cr:main_cor}
Let $X$ be a smooth proper variety such that
$\Bs| \omega_X |$ is a finite set (possibly empty).
Then $\bfD ^{ b } \lb \coh X \rb$ has no non-trivial SOD.
\end{corollary}

\noindent
In particular the global generation of the canonical bundle
implies the non-existence of non-trivial SODs.
Examples of such varieties are submanifolds of abelian varieties (\cite[Lemma 3.11]{MR0360582})
and complete intersections with non-negative Kodaira dimensions in projective spaces.
This is a far generalization of \cite[Theorem 1.1]{MR2838062}.

The other constraint on SODs is the rigidity under the action of
topologically trivial autoequivalences.

\begin{theorem}[{$=$ a special case of \pref{th:rigidity}}]
Let $X$ be a smooth projective variety over $\bfk$ and
$
\bfD ^{ b } \lb \coh X \rb
=
\la \cA, \cB \ra
$
an SOD.
Then for any line bundle
$
 L
$
satisfying
$
 [ L ] \in \PPic _{ X / \bfk }^0
$, we have the equality of subcategories
$
 \cA \otimes L = \cA \subset \bfD ^{ b } \lb \coh X \rb.
$
\end{theorem}

\noindent
Immediately we obtain
\begin{corollary}
[$=$ \pref{cr:trivial_chern_class}]
Let $X$ be a smooth projective variety satisfying
$
 [ \omega_X ] \in \PPic _{ X / \bfk }^0
$.
Then $\bfD ^{ b } \lb \coh X \rb$ admits no SOD.
\end{corollary}

\begin{corollary}\label{cr:the_converse_does_not_hold}
Let $X$ be the product of a bielliptic surface with an abelian variety.
Then $\bfD ^{ b } \lb \coh X \rb$ admits no SOD.
\end{corollary}

\noindent
Since the canonical bundle of the variety $X$ in
\pref{cr:the_converse_does_not_hold}
admits no global section (in another word
$
\Bs | \omega_X |
=
X
$),
the converse of \pref{cr:main_cor} does not hold at all in dimensions at least two.

In \pref{sc:Surfaces} we closely study SODs of minimal surfaces.
We obtain satisfactory results for the cases
$
 \kappa ( X )
 =
 0
$
and
$1$,
where
$
 \kappa ( X )
$
is the Kodaira dimension of $X$. Below is the summary.

\begin{theorem}
Let $X$ be a smooth projective minimal surface.
\begin{enumerate}
\item
If $\kappa(X)=0$, then $\bfD ^{ b } \lb \coh X \rb$ admits a non-trivial SOD if and only if $X$ is a classical Enriques surface.

\item
If $\kappa(X)=1$ and $p_g (X) > 0$, then
$\bfD ^{ b } \lb \coh X \rb$ admits no SOD.
\end{enumerate}
\end{theorem}

\noindent
In the study of SODs of (quasi-)elliptic fibrations (\pref{sc:kappa=1}),
\pref{th:rigidity} will be effectively used on multiple fibers.

For the case $\kappa(X)=2$ we have to put a rather strong assumption
to prove the non-existence of SODs. We believe that it is a technical assumption and
should be eventually removed.

\begin{theorem}[$=$ \pref{th:negative_definite}]
Let $X$ be a minimal smooth projective surface of general type with
$
\dim_{\bfk} H^0(X, \omega_X) > 1
$.
Assume the following condition (*).
\begin{quote}
(*)
For any one-dimensional connected component $Z \subset \Bs| \omega_X |$,
its intersection matrix is negative definite.
\end{quote}
Then $\bfD ^{ b } \lb \coh X \rb$ admits no SOD.
\end{theorem}

\noindent
If there exists a local section of $\omega_X$ defined on an infinitesimal neighborhood
of $\Bs| \omega_X |$, similar arguments as in the proof of \pref{th:pg positive}
work.
See \pref{sc:Local_situation} and \pref{cr:induction}
for the precise statements.
This will be effectively used in the proof of \pref{th:negative_definite}.

Our results for minimal surfaces apparently indicate that
the existence/non-existence of SOD corresponds to the conditions
$p_g>0/=0$. This resembles the (conjectural) dichotomy of infinite/finite generation of the Chow group of zero-cycles on surfaces; i.e., a theorem of Mumford \cite{MR0249428} and the Bloch conjecture \cite{MR0371891}.

In \pref{sc:Toward_higher_dimensions},
we briefly treat varieties of dimensions greater than $2$.
There is nothing new in the case $\kappa=0$,
and we discuss the case $\kappa=1$.
A new difficulty, which does not appear in dimension $2$, then shows up.
We illustrate the issue by an example due to Keiji Oguiso.

Finally, in \pref{sc:twisted_sheaves} we generalize our results to the category of
twisted sheaves. It turns out that our arguments go through without
essential change, though there are a couple of technical issues to be settled.

Our results so far seem to imply the following general principle:
\begin{quote}
For "most" minimal $X$ with non-negative Kodaira dimension, there is no SOD of
$
 \bfD ^{ b } \lb \coh X \rb
$; i.e., the converse of \pref{cj:MMP_implies_SOD_minimal_case} is true.
\end{quote}

The precise meaning of "most", in particular in higher dimensions, is very vague at this moment and remains to be made precise. Note that the principle can be regarded as a special case of the following assertion:
\begin{quote}
For "most" $X$ with non-negative Kodaira dimension, 
there are no SODs of
$
 \bfD^{b} \lb \coh X \rb
$
other than those originating from the MMPs starting from $X$.
\end{quote}
This generalized assertion can be checked for some non-minimal surfaces, and it will be discussed in future work(s).

\begin{not_and_conv}
The base field $\bfk$ will be assumed to be algebraically closed, without loss of generality. In fact, if $\bfk$ is not algebraically closed, any SOD of
$
 \bfD ^{ b } \lb \coh X \rb
$
induces an SOD of
$
 \bfD \lb X \otimes_{\bfk} \bfkbar \rb
$,
where $\bfkbar$ is the algebraic closure of $\bfk$, which is invariant under the action of the absolute Galois group
$
 \Aut \lb \bfkbar / \bfk \rb
$
(\cite[Proposition 5.1]{MR2801403}).
Since $X$ is connected, the absolute Galois group
acts transitively on the set of connected components
$
 \pi _{ 0 } \lb X \otimes _{ \bfk } \bfkbar \rb
$.
Therefore in order to show the non-existence of SOD for
$
 \bfD ^{ b } \lb \coh X \rb
$, it is enough to show it for
$
 \bfD ^{ b } ( \Xbar )
$
for a connected component
$
 \Xbar \subset X \otimes_{\bfk} \bfkbar
$.

Any Deligne-Mumford stack in this paper will be assumed to be connected and
smooth over $\bfk$, unless otherwise stated.

The following standard symbols will be used.
\begin{itemize}
\item
(geometric genus)
$
 p _{ g } ( X )
 = \dim H^0 ( X, \omega _X )
 = \dim \Hom _{ \bfD ^{ b } \lb \coh X \rb } ( \cO_X, \omega_X )
$

\item
(irregularity)
$
 q ( X ) = \dim H^1 ( X, \cO _X )
$

\item
(canonical sheaf)
$
 \omega _{ X }
$

\end{itemize}
\end{not_and_conv}

%
%
\begin{acknowledgements*}
The authors would like to thank Yujiro Kawamata and Mihnea Popa
for their corrections and comments about stacks,
and Keiji Oguiso for useful discussions and providing them with the example.
They are indebted to Marcello Bernardara for suggesting generalization to
twisted sheaves, and Kazuhiro Konno for informing them of folklore conjectures
on the canonical linear systems of surfaces of general type
and for many valuable discussions.
They would also like to thank Alexey Bondal, Andrei
C\u{a}ld\u{a}raru, Daniel Huybrechts, and Alexander Kuznetsov for stimulating discussions about support of
objects in the derived category.
They would also like to thank Pawel Sosna for carefully reading the first draft
and providing them with many corrections and suggestions,
David Favero for communicating to them the paper
\cite{rosay2009some}, and Hiroyuki Minamoto for informing them of the concise proof of \pref{lm:classical_generator} which is reproduced below in this paper.

This work was initiated while the second author was visiting the University of Michigan. He would like to thank Mircea Musta\c{t}\u{a} for his support and hospitality.
The first author was partially supported by Grant-in-Aid for Scientific Research (S) (No. 22224001),
and the second author was partially supported by JSPS fellowship (22-849),
Grant-in-Aid for Scientific Research (
60646909,
16H05994,
16K13746,
16H02141,
16K13743,
16K13755,
16H06337)
and the Inamori Foundation.
\end{acknowledgements*}

%
%

\section{Preliminaries}

\subsection{Deligne-Mumford stack}

We begin with a couple of notions about Deligne-Mumford (DM for short) stacks.

\begin{definition}
Let $X$ be a DM stack. A \emph{closed point} of $X$ is an irreducible reduced closed substack
of dimension zero.
For a closed point
$
 \iota \colon x \hto X
$, the sheaf $\iota_{*}\cO_x$ will be denoted by $\bfk ( x )$.
This situation will be simply denoted by $x\in X$.
\end{definition}

\begin{definition}\label{df:base_locus}
Let $X$ be a smooth DM stack. The
\emph{base locus of the complete canonical linear system}
$
 \Bs | \omega_X |
$
is the closed substack of $X$ defined by the ideal
\begin{align}
 \Image \lb \Hom ( \omega _{ X } ^{ - 1 }, \cO _X ) \otimes \omega _{ X } ^{ - 1 }
 \xto[]{\ev}
 \cO _X \rb \subset \cO _X.
\end{align}
\end{definition}
See \cite[Application (14.2.7)]{MR1771927} for the correspondence between
closed substacks and coherent ideal sheaves.

\begin{definition}\label{df:stacky_locus}
Let $X$ be a smooth proper DM stack. By \cite[Corollary 1.3.(1)]{MR1432041},
$X$ admits a coarse moduli algebraic space
$
 \pi \colon X \to |X|
$.
The \emph{stacky locus}, which will be denoted by
$\cS\subset X$,
is the complement of the maximal open substack of $X$ on which
$\pi$ is an isomorphism to its image.
A closed point
$
 \iota \colon x \hookrightarrow X
$
is said to be \emph{non-stacky} if
$
 \iota
$
factors through
$
 X \setminus \cS
$.
\end{definition}

\begin{definition}\label{df:support}
Let $X$ be a smooth DM stack, and $E\in \bfD ^{ b } \lb \coh X \rb$ a bounded complex of
coherent sheaves on $X$. The \emph{support of $E$}
is the union of the supports of its cohomology sheaves: i.e.,
$
 \Supp{E} = \bigcup_i \Supp{ \cH ^i(E)} _{ \mathrm{red} }.
$
By definition,
$
 \Supp{  E  }
$
is a closed substack of
$
 X
$.
If
$
 \Supp  E 
$
is contained in a closed substack
$
 Z \subset X
$,
we say that\emph{
$
 E
$
is supported in
$
 Z
$}.
\end{definition}

\begin{remark}\label{rm:rouquier}
The support of $E$ can be alternatively defined as the
complement of the maximal open substack of $X$ on which
$E$ is zero.

If $E$ is a coherent sheaf, we can introduce the natural stack structure on
$
 \Supp E
$
(which is not necessarily reduced)
in such a way that $E$ is isomorphic to
the pushforward of a coherent sheaf on $\Supp E$ (see \cite[Section 1.1]{Huybrechts-Lehn}).
In general, let $Z\subset X$ be a reduced closed substack and
$
 E \in \bfD ^{ b } \lb \coh X \rb
$
a bounded complex of coherent sheaves supported in $Z$.
Then by \cite[Lemma 7.40]{MR2434186}
there exists a positive integer $n>0$ and a
bounded complex of coherent sheaves
$
 E ' \in \bfD ^{ b } \lb \coh n Z \rb
$
such that
$
 E \simeq \iota_*E'
$,
where
$
 \iota \colon n Z \hto X
$
is the natural closed immersion.
The proof uses the Artin-Rees lemma in an essential way, so that
we do not have a control over the value of $n$.
If we could introduce a sufficiently thin stack structure on the support of complexes,
that would be useful to improve the results of this paper.
\end{remark}

\begin{lemma}\label{lm:invariant}
Let $X$, $E$, and $Z$ be as in \pref{rm:rouquier}.
Suppose that $L$ is a line bundle on $X$ which is trivial on
an open neighborhood of each connected component of $Z$. Then
$
 E \simeq E\otimes L.
$
\end{lemma}

\begin{proof}
If
$
 Z _{ 1 }, \dots, Z _{ m }
$
are the connected components of $Z$, then there are objects
$
 E _{ i }
$
whose supports are contained in $Z _{ i }$ and
$
 E \simeq \oplus _{ i = 1 } ^{ m } E _{ i }
$.
Therefore we may and will assume that $Z$ itself is connected in the rest of the proof.

Let
$
 \iota \colon U \hto X
$
be an open neighborhood of $Z$ on which
$L$ is trivial. Since $E$ is supported in $Z$, the natural morphism
$
 E \to \iota_*\iota^*E
$
is an isomorphism. The same holds for $ E \otimes L $, and hence
\begin{equation}
 E
 \simeq \iota_*\iota^*E
 \simeq \iota_*(\iota^*E \otimes L|_U)
 \simeq \iota_*\iota^*(E \otimes L)
 \simeq E \otimes L.
\end{equation}
\end{proof}

Next we recall the Serre duality for tame DM stacks. Recall that a DM stack over a field is
\emph{tame} if the order of the stabilizer groups are not divisible by the characteristic of the base field.
In this paper the dualizing sheaf of $X$ will be denoted by $\omega_X$.

\begin{fact}
Let $X\to\Spec \bfk$ be a tame smooth proper DM stack of dimension $n$ over a field $\bfk$.
Then $- \otimes _{ X } \omega_X [ n ]$ is a Serre functor of
$
 \bfD ^{ b } \lb \coh X \rb
$.
\end{fact}

We will use the following useful lemma from \cite[Proposition 1.5]{Bondal-Orlov_semiorthogonal}.
\begin{lemma}\label{lm:useful}
Let $X$ be a DM stack and $x\in X$ a non-stacky closed point.
Take an upper bounded complex of coherent sheaves
$
 E \in \bfD ^{ - } \lb \coh X \rb
$.
If
$
 x \in \Supp E
$, then
$
 \Hom ( E, \bfk ( x ) [ i ]  ) \neq 0
$
for some
$
 i \in \bZ
$. Moreover if $X$ is tame, smooth, and proper, then
$
 \Hom ( \bfk ( x ) [ j ] , E ) \neq 0
$
also holds for some
$
 j \in \bZ
$.
\end{lemma}
\begin{proof}
Since the support of an object in invariant under a tensor product with
an invertible sheaf, one can reduce the second assertion
to the first one by applying the Serre functor to the object $E$.

Suppose that the closed point $x$ is contained in $\Supp E$. 
Set
$
 m = \max \lc j \mid x \in \Supp H ^{ j } ( E )  \rc.
$
We have a distinguished triangle 
$E^{\leq m} \to E \to E^{>m} \to E^{\leq m[1]}$, 
where $E^{>m}$ (resp. $E^{\leq m}$) is the upper (resp. lower) truncation of $E$.
Since
$
 x \nin \Supp \lb E^{>m} \rb
$, we have $\Hom(E^{>m}, \bfk(x)[p] ) =0$ for all $p \in \bZ$. 
Hence we see $\Hom(E, \bfk(x)[-m] ) \simeq \Hom(E ^{\leq m}, \bfk(x)[-m])$. 
Moreover we have 
\begin{equation}\label{eq:support_closed_point}
\Hom(E ^{\leq m}, \bfk(x)[-m]) \simeq \Hom (H^m(E)[-m], \bfk(x)[-m]),
\end{equation}
since
$
 \Hom \lb (E^{\leq m})^{\leq m-1}, \bfk(x)[-m] \rb
 =
 \Hom \lb E^{\leq m - 1}, \bfk(x)[-m] \rb
 =
 0
$
by the degree reason. Since the right hand side of \eqref{eq:support_closed_point} is non-zero, we conclude the proof.
\end{proof}

For the sake of completeness, here we include the following lemma.

\begin{lemma}[Nakayama-Azumaya-Krull]\label{lm:Derived_NAK}
Let $X$ be a scheme,
$
 \iota \colon x \hto X
$
a closed point, and
$
 E \in \bfD ^{ - } ( \coh X )
$
a complex of coherent sheaves on $X$ bounded above. If
$
 \bL \iota^*E = 0
$, then
$
 x \nin \Supp{E}
$.
\end{lemma}
\begin{proof}
See \cite[Lemma 3.29 and Exercise 3.30]{MR2244106}.
\end{proof}

%
%

\subsection{(Semi)orthogonal decomposition}

We recall the notion of (semi)orthogonal decompositions and show
the non-existence of orthogonal decompositions for stacks which admits
a non-stacky closed point. This is well known for varieties
(see, e.g., \cite[Proposition 3.10]{MR2244106}).
The proof given below was suggested by Yujiro Kawamata.
Recall that a full subcategory
$
 \cA \subset \cC
$
is said to be \emph{strict}
if any object in
$
 \cC
$
which is isomorphic to an object in
$
 \cA
$
is already contained in
$
 \cA
$.

\begin{definition}\label{df:SOD}
A pair of strictly full triangulated subcategories $\cA,\cB$ of a triangulated category
$\cT$ is a \emph{semiorthogonal decomposition} if the following conditions are satisfied:
\begin{itemize}
\item $\Hom_{\cT}(b,a)=0$ for all $a\in\cA$ and $b\in\cB$.
\item Any object $x\in\cT$ is decomposed into a pair of objects
$a\in\cA$ and $b\in\cB$ by a distinguished triangle
\begin{equation}\label{eq:decomposing_triangle}
 b \to x \to a \to b [ 1 ].
\end{equation}
\end{itemize}
This situation will be denoted by the symbol
$
 \cT = \la \cA, \cB \ra
$.
If $\Hom_{\cT}(\cA, \cB)=0$ also holds, the decomposition
is called an \emph{orthogonal decomposition} (OD for short) and denoted by
$
 \cT = \cA \oplus \cB
$.
In this case the triangle \eqref{eq:decomposing_triangle}
splits and we obtain the direct sum decomposition
$
 x \simeq a \oplus b
$.
\end{definition}

\begin{remark}
\begin{enumerate}
\item If
$
 \cT = \la \cA,\cB \ra
$
is an SOD of $\cT$, then
$
 \cA = \cB^{\perp}
$
and
$
 \cB ={}^{\perp} \hspace{-1mm}\cA
$.
Here
$
 \bullet^\perp
$
(resp.
$
 {}^\perp \bullet
$)
denotes the right (resp. left) orthogonal complement of the subcategory
$\bullet$ (\cite[Definition 1.42]{MR2244106}).

\item For a strictly full triangulated subcategory
$
 \iota \colon \cC \subset \cT
$,
the pair
$
 \cC, {}^{\perp} \cC
$
(resp. $\cC^{\perp}, \cC$)
gives an SOD of $\cT$
if and only if the inclusion functor $\iota$ admits a left (resp. right) adjoint (see \cite[Lemma 3.1]{Bondal_RAACS}).
\end{enumerate}
\end{remark}

\begin{lemma}\label{lm:no_OD}
Let $X$ be a connected locally separated DM stack which admits a non-stacky point.
Assume that $X$ satisfies the resolution property.
Then
$
 \bfD ^{ - } \lb \coh X \rb,
 \bfD ^{ b } \lb \coh X \rb
$, and
$
 \bfD ^{ \perf } \lb X \rb
$
admit no non-trivial OD.
\end{lemma}

In the statement above,
$
 \bfD ^{ - } \lb \coh X \rb
$
and
$
 \bfD ^{ \perf } \lb X \rb
$
denote the category of upper bounded complexes of coherent sheaves and those of perfect complexes, respectively.

\begin{proof}
Let
$
 \bfD ^{ b } \lb \coh X \rb = \cA \oplus \cB
$
be an OD of $\bfD ^{ b } \lb \coh X \rb$.
Take a non-stacky point $x\in X$.
Since $\End(\bfk ( x ))$ is a field,
$\bfk ( x )$ is indecomposable and hence is contained in either $\cA$ or $\cB$.
Let us assume it is contained in $\cA$.

Let $E$ be any locally free sheaf, and consider the decomposition
$
 E = E_{\cA} \oplus E_{\cB}
$
provided by the OD. If
$E_{\cB} \neq 0$, then
$
 E_{\cB} \otimes \bfk ( x )
$
is isomorphic to
$
 \bfk ( x )^{\oplus r}
$ with
$
 r
$ the rank of $E_{\cB}$.
This follows from the connectedness of $X$
and the fact that the closed substack $x$ is a scheme \cite[Chapter II, Proposition 5.9]{MR0302647}.
Thus we get a surjective morphism
$
 E_{\cB}\to E _{ \cB } \otimes \bfk ( x )
 \to \bfk ( x )
$ 
and it is a contradiction. Hence $E$ belongs to $\cA$.
Since locally free sheaves form
a spanning class of $\bfD ^{ b } \lb \coh X \rb$ because of the resolution property. To see this, take any non-trivial object
$
 F ^{ \bullet } = \lb F ^{ i } \rb _{ i \in \bZ }
$
and set
$
 M \coloneqq \max \lc i \in \bZ \mid \cH ^{ i } \lb F ^{ \bullet } \rb \neq 0 \rc
$.
By replacing $F ^{ \bullet }$ with its canonical truncation at $M$, we may assume that
$
 F ^{ i } \neq 0
$
for
$
 i > M
$.
Then cover $F ^{ M }$ by a locally free sheaf $E ^{ M }$ and extend it to a  locally free sheaf resolution
$
 E ^{ \bullet } = \lb E ^{ i } \rb _{ i \in \bZ }
$
of
$
 F ^{ \bullet }
$.
Then one obtains a non-trivial map
$
 E ^{ M } [ - M ] \to E ^{ \bullet }
$.
Therefore $\cB$ should be trivial by the orthogonality.

Finally, note that the arguments above works perfectly well for
$
 \bfD ^{ - }\lb \coh X \rb
$
as well. This, in turn, implies the claim for
$
 \bfD ^{ \perf }\lb  X \rb
$.
In fact, suppose that there is an OD
$
 \bfD ^{ \perf }\lb  X \rb = \cA \oplus \cB
$.
By \cite[Proposition 4.3]{MR2801403}, there is a unique SOD
$
 \bfD ^{ - } \lb \coh X \rb= \la \cA ^{ - }, \cB ^{ - } \ra
$
which is compatible with the SOD
$
 \bfD ^{ \perf } \lb  X \rb= \la \cA, \cB \ra
$.
Closely looking on the explicit description of the subcategory
$
 \cA ^{ - }
$
(respectively
$
 \cB ^{ - }
$)
(see the proof of \cite[Proposition 4.3]{MR2801403}),
one finds that these categories depend only on the subcategory $\cA$
(resp. $\cB$).
Hence it follows that
$
 \bfD ^{ - }\lb \coh X \rb = \la \cB ^{ - }, \cA ^{ - } \ra
$
is also an SOD which is compatible with the SOD
$
 \bfD ^{ \perf } \lb  X \rb= \la \cB, \cA \ra
$.
Thus we see
$
 \bfD ^{ - } \lb \coh X \rb= \cA ^{ - } \oplus \cB ^{ - }
$,
so that either
$
 \cA ^{ - } = 0
$
or
$
 \cB ^{ - } = 0
$.
Since
$
 \cA ^{ \perf }
$
(resp.
$
 \cB ^{ \perf }
$)
is a full subcategory of
$
 \cA ^{ - }
$
(resp.
$
 \cB ^{ - }
$), we conclude the proof.
\end{proof}

\begin{remark}
If there is no non-stacky point on $X$, $\bfD ^{ b } \lb \coh X \rb$ may have an OD even if
it is smooth and irreducible.
For example consider the quotient stack
$
 X = [ \Spec{ \bC } / ( \bZ / 2 ) ]
$.
A coherent sheaf on $X$ is nothing but a finite dimensional representation of
$\bZ / 2$ over
$
 \bC
$. The derived category $\bfD ^{ b } \lb \coh X \rb$ is orthogonally
decomposed by the trivial representation and the non-trivial character of $\bZ / 2$.
\end{remark}

\pref{lm:no_OD} provides us with the following useful criterion for the triviality of an SOD.
\begin{corollary}\label{cr:iff}
Let $X$ be a smooth proper DM stack which has a non-stacky point and satisfies
the resolution property. Consider an SOD
$
 \bfD ^{ b } \lb \coh X \rb = \la \cA,\cB\ra.
$
Then it is trivial, i.e. $\cA=0$ or $\cB=0$,
if and only if
$
 \cA \otimes \omega_X \subset \cA \subset \bfD ^{ b } \lb \coh X \rb
$
holds.
\end{corollary}

\noindent
The following argument is well known to experts (see \cite[Proposition 3.6]{Bondal-Kapranov_Serre}),
but we include it here because of its importance in this paper.

\begin{proof}
By \pref{lm:no_OD}, it is enough to show that it is an OD. 
Given $a\in\cA$ and $b\in\cB$, by applying the Serre duality and the assumption
$
 a \otimes \omega_X [ \dim X ] \in \cA
$, we see
$
 \Hom ( a, b ) \simeq \Hom ( b, a \otimes \omega_X [ \dim X ] )^{\vee} = 0.
$
Hence we see that $\cA$ is also the right orthogonal of $\cB$, concluding the proof.
\end{proof}

%
%

\subsection{Picard scheme}

We recall basics of Picard schemes from
\cite[Chapter 9]{MR2222646}.
Let $X\to S$ be a morphism of finite type between locally Noetherian schemes.
The relative Picard functor,
which will be denoted by $\Pic_{X/S}$,
is a contravariant functor from the category of locally Noetherian $S$-schemes
to the category of abelian groups defined by
\begin{equation}
 \Pic_{ X / S } ( T ) = \Pic ( X \times _{ S } T ) / \Pic ( T ),
\end{equation}
where $T$ is a locally Noetherian $S$-scheme.
The functor
$
 \Pic _{ X / S }
$
is a presheaf, and  the associated sheaf on the fppf site will be denoted by
$
 \Pic _{ ( X / S ) ( \mathrm{ fppf } )}
$.
If
$
 \Pic _{ ( X / S )( \mathrm{fppf} ) }
$
is represented by a scheme, it will be denoted by
$
 \PPic _{ X / S }
$.
A line bundle $L$ on $X$ naturally defines an $S$-valued point
$
 [ L ] \in \PPic _{ X / S } ( S ).
$

The following existence result for Picard schemes is enough for us.

\begin{theorem}[{$=$\cite[Corollary 9.4.18.3]{MR2222646}}]\label{th:existence}
Let $S$ be the spectrum of a field and $X$ a proper scheme over $S$.
Then
$
 \PPic _{ X / S }
$
exists and is a disjoint union of open quasi-projective subschemes.
\end{theorem}

If $\PPic_{X/S}$ exists and $S$ is the spectrum of a field,
its identity component
(i.e., the connected component containing the identity)
will be denote by $\PPic^0_{X/S}$.
The following theorem characterizes the $\bfk$-valued points of $\PPic^0_{X/\bfk}$

\begin{theorem}[{$=$\cite[Corollary 9.5.10]{MR2222646}}]\label{th:characterization_of_points_of_Pic0}
Assume that $S$ is the spectrum of a field and
$
 \PPic _{ X / S }
$
exists. Let $L$ be an invertible sheaf on X.
Then $L$ is algebraically equivalent to $\cO_X$
if and only if
$
 [ L ] \in \PPic^0 _{ X / S } ( S ).
$
\end{theorem}

The notion of algebraic equivalence is defined as follows.
For simplicity, we assume that the base field is algebraically closed.

\begin{definition}[{$=$\cite[Definition 9.5.9]{MR2222646}}]\label{df:algebraic_equivalence}
Assume $S$ is the spectrum of an algebraically closed field $\bfk$.
Let $L$ and $N$ be invertible sheaves on $X$.
Then $L$ is said to be \emph{algebraically equivalent to} $N$ if, for
some $n$ and all $i$ with $1\le i\le n$, there exist a connected $\bfk$-schemes of finite type
$T_i$, closed points
$
 s_i, t_i \in T_i
$,
and an invertible sheaf $M_i$ on
$
 X \times _{ \bfk } T_i
$
such that
\begin{equation}
 L
 \simeq M _{ 1, s_1 },
 M _{ 1, t_1 }
 \simeq M _{ 2, s_2 },
 \dots,
 M _{ n-1, t _{n-1} }
 \simeq M_{ n, s_n },
 M_{n, t_n}
 \simeq N.
\end{equation}
\end{definition}

%
%

\section{Results in arbitrary dimensions}
\subsection{Constraints on SOD by canonical base loci}

\begin{theorem}\label{th:pg positive}
Let $X$ be a smooth proper DM stack and
$
\bfD ^{ b } \lb \coh X \rb = \la \cA, \cB \ra
$
an SOD. Then

\begin{enumerate}
	\item\label{it:closed points}
	At least one of the followings holds.
	\begin{enumerate}
	\item \label{it:closed points in A}
	$
	\bfk ( x ) \in \cA
	$
	for any closed point
	$
	x \nin \Bs| \omega_X | \cup \cS
	$.
	\item \label{it:closed points in B}
	$
	\bfk (x) \in \cB
	$
	for any closed point
	$
	x \nin \Bs| \omega_X | \cup \cS
	$.
	\end{enumerate}

	\item \label{it:small supports}
	When (\pref{it:closed points in A}) (resp. (\pref{it:closed points in B}))
	is satisfied, the support of any object in $ \cB $ (resp. $ \cA $)
	is contained in
	$ \Bs| \omega_X | \cup \cS $.
	
\end{enumerate}
\end{theorem}

\noindent
In the statement above, $\Bs| \omega_X |$ denotes the canonical base locus (see \pref{df:base_locus}) and
$
\cS \subset X
$
the locus of stacky points (see \pref{df:stacky_locus}).

\begin{proof}
Take an arbitrary closed point
$
 x \in X \setminus \lb \Bs | \omega _X | \cup \cS \rb
$.
We first show that
$
 \bfk ( x )
$
is contained in either $\cA$ or $\cB$. 
Let us include $\bfk ( x )$ in the triangle provided by the SOD:
\begin{equation}
 b \to \bfk ( x ) \to a \xto[]{f}  b [ 1 ]. 
\end{equation}

\noindent
Take a global section
$
 s \in H^0 ( X, \omega _X )
$
which is not vanishing at $x$, and set
$
 U = X \setminus Z ( s )
$.
If
$
 f| _U \neq 0
$, we see
$
 ( \otimes s \circ f ) | _U
 = \otimes s| _U \circ f|_ U
$
is also non-trivial. This contradicts
\begin{equation}
 \Hom ( a, b \otimes \omega _X [ 1 ] )
 \simeq \Hom ( b, a [ \dim X - 1 ] ) ^{\vee}
 = 0.
\end{equation}
Thus we see
$
 f| _U = 0
$.
This implies the decomposition
$
 \bfk ( x )
 \simeq a| _U \oplus b| _U
$
and hence we obtain either
$
 a| _U = 0
$
or
$
 b| _U = 0
$.
If
$
 a| _U = 0
$,
the morphism
$
 \bfk ( x ) \to a
$
is zero. Then we obtain the decomposition
$
 b \simeq \bfk ( x ) \oplus a [ - 1 ]
$.
By the semiorthogonality we see
$
 a = 0
$
and hence
$
 \bfk ( x ) \in \cB.
$
If we instead assume
$
 b | _{ U } = 0,
$
similarly we obtain
$
 \bfk ( x ) \in \cA.
$

If $\bfk ( x ) \in \cB$ (respectively $\bfk ( x ) \in \cA$) holds for some
closed point
$
 x \nin \cS
$,
then by \pref{lm:useful} any object $E\in\cA$
should satisfy
$
 x \nin \Supp{E}
$ (resp. any $ E \in \cB $).
This in particular implies that the closed substack
$
 \Supp{ E } \subset X
$
is strictly smaller than $X$.

Finally assume for a contradiction that $\cA$ and $\cB$
both contain non-stacky closed points. As said in the previous paragraph,
the support of any object in $\cA$ or $\cB$ is a strictly smaller closed subset of $X$. 
On the other hand, consider the decomposition of the structure sheaf
\begin{equation}
 b \to \cO_X \to a \to b [ 1 ].
\end{equation}
From this triangle we obtain the equality
$
 X = \Supp { \cO_X } = \Supp { a } \cup \Supp { b },
$
which contradicts the irreducibility of $X$.
Hence we see that all the non-stacky closed points outside of $\Bs| \omega_X |$
are contained simultaneously in $\cA$, or otherwise in $\cB$.
This concludes (\pref{it:closed points}) of \pref{th:pg positive}.
In the former case, the support of any object in
$\cB$ is contained in $\cS\cup\Bs| \omega_X |$ as we saw above, concluding
the proof of (\pref{it:small supports}).
\end{proof}

\begin{example}
Let $Y$ be a smooth projective surface such that
$
 \omega _{ Y }
$
is globally generated. Let
$
 f \colon X \to Y
$
be the blow-up of $Y$ at a closed point $y$, with the exceptional divisor
$
 E \subset X
$.
Then we obtain the SOD
$
 \bfD ^{ b } \lb \coh X \rb = \la \la \cO _E ( E ) \ra, \bL f^{*} \bfD ( Y ) \ra
$
(see \cite{Orlov_PB}).
Observe that the objects in
$
 \la \cO _E ( E ) \ra
$ are supported in
$
 E = \Bs | \omega_X |
$,
and that all the closed points
$
 x \nin \Bs | \omega_X |
$ are contained in
$
 \bL f^{*} D ( Y )
$.
\end{example}

As an immediate consequence of \pref{th:pg positive}, we obtain the following general result for the non-existence of SODs.

\begin{theorem}\label{th:locally free}
Let $X$ be a smooth proper DM stack satisfying the following properties:
\begin{enumerate}
\item there exists a non-stacky closed point.

\item $X$ satisfies the resolution property: i.e., every coherent sheaf on $X$
admits a surjective morphism from a locally free sheaf.
\label{it:resolution_property}

\item Each connected component of $\cS\cup\Bs| \omega_X |$ is contained in an open substack of $X$ on which
$\omega_X$ is trivial.
\end{enumerate}
Then $\bfD ^{ b } \lb \coh X \rb$ admits no non-trivial SOD. 
\end{theorem}

\noindent
If the coarse moduli space of $X$ is projective, then the condition \eqref{it:resolution_property} is always satisfied by \cite[Theorem 4.2]{Kawamata_EDCSSS}. Actually, it is enough to assume that the coarse moduli is a scheme with affine diagonal (\cite[Theorem 1.2]{MR2108211}).

In order to establish the correspondence between MMP and SOD for varieties with quotient singularities, one should think of the derived category of the smooth DM stack which is obtained by replacing the quotient singularity with the corresponding stacky quotient (see \cite{Kawamata_LCBMDC}). This is one of the reasons why we should think of stacks, not only schemes.

\begin{proof}
Take an SOD
$
 \bfD ^{ b } \lb \coh X \rb = \la \cA, \cB \ra
$.
Write
$
 U = X \setminus \lb \cS \cup \Bs| \omega _X |  \rb
$.
By \pref{th:pg positive}, closed points of $U$
are simultaneously contained in either $\cA$ or $\cB$.
Let us assume they are in
$
 \cB.
$

By \pref{th:pg positive}, then the support of any object
$
 a \in \cA
$
is contained in
$
 \cS \cup \Bs | \omega _X |
$.
Since
$
 \omega _X
$
is trivial on an open neighborhood of each connected component of this set, we see
$
 a \otimes \omega_X \simeq a
$
by \pref{lm:invariant}.
Therefore we can apply \pref{cr:iff} to conclude the proof.
\end{proof}

\begin{example}
Let $X$ be the surface (of general type) discussed in \cite[Proposition 3]{MR1980618}.
We can easily check that
$
 \Bs| \omega_X |
$
consists of $4$ points. Hence
\pref{cr:main_cor} tells us that the derived category $\bfD ^{ b } \lb \coh X \rb$ admits no SOD.
\end{example}

%
%

\subsection{Local situation}\label{sc:Local_situation}

We refine the arguments in the proof of \pref{th:pg positive},
so as to make it applicable to local situations.
This will be applied later to surfaces of general type.
For simplicity we restrict ourselves to varieties.

Let $X$ be a variety and $Z$ a closed subset. Consider the strict full subcategory
\begin{equation}\label{eq:category_supported_on_Z}
 \cC_Z
 = \{ E \in \bfD ^{ b } \lb \coh X \rb | \Supp{E} \subset Z \} \subset \bfD ^{ b } \lb \coh X \rb.
\end{equation}

\noindent
An immediate corollary of \pref{th:pg positive} is
\begin{lemma}\label{lm:restriction}
Let $X$ be a smooth proper variety with
$
 p_g(X) > 0.
$
Suppose that for any connected component $Z$ of $\Bs| \omega_X |$,
the category $\cC_Z$ admits no SOD. Then $\bfD ^{ b } \lb \coh X \rb$ has no SOD.
\end{lemma}

\begin{proof}
Take an SOD
$
 \bfD ^{ b } \lb \coh X \rb = \la \cA, \cB \ra
$.
By \pref{cr:iff}, it is enough to show
$
 \cA \otimes \omega_X
 =
 \cA
$.
Assume that the conclusion \eqref{it:closed points in B} of
\pref{th:pg positive} holds, so that $\cA$
is a triangulated subcategory of
$
 \cC _{ \Bs | \omega _X | }
$.
For any connected component
$
 Z \subset \Bs | \omega_X |
$,
set
$
 \cA_Z = \{ a \in \cA | \Supp{a} \subset Z \}
$
so as to obtain the orthogonal decomposition
$
 \cA = \bigoplus _{Z} \cA_Z
$
by the triangulated subcategories
$
 \cA_Z \subset \cC_{Z}
$.

Let $p\colon \bfD ^{ b } \lb \coh X \rb\to\cA$ be the left adjoint of the inclusion functor
$
 \cA \hto \bfD ^{ b } \lb \coh X \rb
$.
Composing $p$ with the obvious functors,
we get the left adjoint of the inclusion
$
 \cA_Z\hto \cC_Z
$
and hence obtain the SOD
$
 \cC_Z = \la \cA_Z, {}^\perp \hspace{-1mm} \cA_Z \ra.
$
Since
$
 \cC_Z
$
admits no SOD by the assumption,
$
 \cA_Z
$
is either
$
 0
$
or
$
 \cC _Z
$ itself. In any case we obtain the equality 
$
 \cA_Z\otimes \omega_X = \cA_Z \subset \bfD ^{ b } \lb \coh X \rb.
$
Summing up over $Z$, we obtain the conclusion.
\end{proof}

Now we show the local version of \pref{th:pg positive}
for $\cC_Z$.

\begin{proposition}\label{pr:closed_points_local}
Let $X$ be a smooth proper variety and $Z \subset X$ a closed subscheme.
Assume that for each $m \ge 1 $
we have a section
$
 s_m \in H^0 ( mZ, \omega _{ X } |_{ mZ } )
$
such that
$
 s _{ m + 1 } | _{ m Z } = s _{ m }
$
and
$
 s_1
$
is generically non-vanishing on an irreducible component
$
 W \subset Z
$. Write the projective limit as
$
 s = ( s_m )_{ m \ge 1 }
 \in
 \varprojlim H^0 ( mZ, \omega _X | _{ m Z } ).
$

Then for any SOD $\cC_Z=\la \cA, \cB \ra$,
one and only one of the followings holds.
\begin{enumerate}[(a)]
\item For any closed point $x\in W$ at which $s$ does not vanish, $\bfk ( x )\in\cA$.
\item For any closed point $x\in W$ at which $s$ does not vanish, $\bfk ( x )\in\cB$.
\end{enumerate}
\end{proposition}

\begin{proof}
We follow essentially the same line as that of the proof of \pref{th:pg positive}.
Take any closed point $x\in W$ at which $s$ does not vanish.
Consider the decomposition
\begin{equation}
 b \to \bfk ( x ) \to a \xto[]{f} b [ 1 ]. 
\end{equation}

\noindent
Since $b$ is supported in $Z$,
by \pref{rm:rouquier} there exists $ m \ge 1 $ and
$
 b' \in D ( m Z )
$
together with an isomorphism
$
 b \simto \iota_* b',
$
where
$
 \iota \colon mZ \hto X
$
is the natural immersion. Hence we can define the ``multiplication by $s$'' as the
composition of morphisms
\begin{equation}
 b \simto \iota_* b'
 \xto[]{ \iota_* \otimes s_m }
 \iota _* ( b' \otimes \omega _{ X } | _{ m Z } )
 \simto
 \iota_* b' \otimes \omega_X
 \simto
 b \otimes \omega_X. 
\end{equation}
Arguing as in the proof of \pref{th:pg positive}, we can show that
the morphism $f$ vanishes at $x$. Thus we see that $\bfk ( x )$ is contained in
$\cA$ or $\cB$.

Finally, by looking at the decomposition of $\cO_W$ instead of $\cO_X$,
we see that all such closed points are simultaneously contained in $\cA$
or $\cB$.
\end{proof}

The next statement provides us with an inductive way of proving the non-existence of SODs.

\begin{corollary}\label{cr:induction}
Let $X$ be a smooth proper variety and $Z$ a connected closed subscheme.
Define the closed subset $B\subset Z$ by
\begin{equation}
 B
 = \bigcap _{ s \in \varprojlim H^0 ( mZ, \omega _{ X } | _{ mZ })} V ( s _1 ).
\end{equation}
Assume 
$
 B \neq Z
$,
and take an irreducible component $Z_1\subset Z$
which is not contained in $B$.
Then $\cC_Z$ admits no SOD if the same holds
for all $\cC_W$, where $W$ runs through all the connected
components of $\overline{(Z\setminus Z_1)}\cup (B\cap{Z_1})$.
\end{corollary}

\begin{proof}
Note first that $\cC_Z$ admits no OD, since it is connected; proof is
essentially the same as that of \pref{lm:no_OD}, once one replaces `locally free sheaves'
with `locally free sheaves on thickenings of $Z$'.
Hence it is enough to show that any SOD of $\cC_Z$ is in fact an OD.

Take an SOD
$
 \cC _Z = \la \cA, \cB \ra.
$
By \pref{pr:closed_points_local}, one and only one of the followings holds:
\begin{enumerate}[(a)]
\item For any closed point $x\in Z_1 \setminus B$, $\bfk ( x )\in\cA$.
\item For any closed point $x\in Z_1 \setminus B$, $\bfk ( x )\in\cB$.
\end{enumerate}
Let us assume (a) holds. Then, as before, for any
$
 b \in \cB
$
we see
$
 \Supp b \subset
 \overline{ \lb Z \setminus ( Z_1 \setminus B )  \rb }
 = \overline{ ( Z \setminus Z_1 ) } \cup ( B \cap Z_1 )
$.
The rest of the proof is completely analogous to that of \pref{lm:restriction}.
\end{proof}

\begin{remark}
In fact, the arguments above work under weaker assumptions.
It is enough to find infinitely many integers
$
 m > 0
$
such that for each
$
 m
$
one can find
$
 s _m \in H^0 ( m Z, \omega_X | _{ m Z } )
$
which does not vanish at the generic point of an irreducible component of
$
 Z.
$
\end{remark}

%
%

\subsection{Rigidity of semiorthogonal decomposition}
\label{sc:Rigidity_of_Semiorthogonal_decomposition}
We show that SODs are rigid under the actions of topologically trivial autoequivalences.
\begin{theorem}\label{th:rigidity}
Let $X$ be a projective scheme over a field
$
 \bfk
$,
and
$
 \bfD ^{ \perf } \lb X \rb
 =
 \la \cA, \cB \ra
$
be an SOD. Then for any line bundle $L$ such that
$
 [ L ] \in \PPic_{X/\bfk}^0
$,
the equality of subcategories
$
 \cA \otimes L= \cA \subset \bfD ^{ \perf } \lb X \rb
$
holds.
\end{theorem}

We immediately obtain
\begin{corollary}\label{cr:trivial_chern_class}
Let $X$ be a smooth projective variety whose
canonical bundle is contained in $\PPic_{X/\bfk}^0$.
Then $\bfD ^{ b } \lb \coh X \rb$ has no SOD.
\end{corollary}
\begin{proof}
By \pref{cr:iff}, it is enough to show that any SOD
$\bfD ^{ b } \lb \coh X \rb=\la \cA, \cB \ra$ satisfies
$
 \cA \otimes \omega_X = \cA
$.
By \pref{th:rigidity},
this follows from the assumption $[\omega_X]\in\PPic^0_{X/\bfk}$.
\end{proof}

\begin{lemma}\label{lm:classical_generator}
Let
$
 X
$
be a quasi-compact and separated scheme and
$
 \cC \subset \bfD ^{ \perf } \lb X \rb
$
be a right or left admissible subcategory. Then
$
\cC
$
admits a classical generator. 
\end{lemma}

See \cite[p. 2]{Bondal-van_den_Bergh} for the definition of classical generator.

\begin{proof}
We first show the existence of classical generator for
$
 \bfD ^{ \perf } \lb X \rb
$.
By \cite[Theorem 3.1.1 (2)]{Bondal-van_den_Bergh},
the derived category
$
 \bfD _{ \Qcoh } \lb X \rb
$
of unbounded complexes of
$
 \cO _{ X }
$-modules with quasi-coherent cohomology is generated by a single compact object
$
 T
$.
Hence by the Ravenel-Neeman theorem \cite[Theorem 2.1.2]{Bondal-van_den_Bergh}
(see also \cite[\href{https://stacks.math.columbia.edu/tag/09SM}{Tag 09SM}]{stacks-project} for a self-contained proof), the subcategory
$
 \bfD _{ \Qcoh } \lb X \rb ^{ c }
$
of compact objects is classically generated by $T$.

On the other hand, it follows from \cite[Corollary 5.5]{MR1214458} that the canonical exact functor
$
 \bfD \lb \Qcoh X \rb \to \bfD _{ \Qcoh } \lb X \rb
$,
where
$
 \bfD \lb \Qcoh X \rb
$
is the derived category of unbounded complexes of quasi-coherent sheaves on $X$,
is an equivalence of triangulated categories. Since
$
 \bfD \lb \Qcoh X \rb ^{ c } = \bfD ^{ \perf } \lb X \rb
$
by \cite[Theorem 3.1.1 (1)]{Bondal-van_den_Bergh}, we obtain the assertion for $\cC =  \bfD ^{ \perf } \lb X \rb$.

%
%
%
%
To obtain a classical generator of general 
$
 \cC
$, 
take the projection of a classical generator of
$
 \bfD ^{ \perf } \lb X \rb
$
by the projection functor. In fact, fix a classical generator
$
 T
$
of
$
 \bfD ^{ \perf } \lb X \rb
$
so that for each object
$
 a \in \cA
$
there is a sequence 
\begin{align}
\xymatrix{
 0 \ar[rr] & & E ^{ 1 } \ar[rr] \ar[dl] & & E ^{ 2 }  \ar[dl] \ar[r] &\cdots \ar[r] & E ^{ n - 1 }
 \ar[rr] \ar[dl] & & E ^{ n } = \iota _{ \cA } \lb a \rb \oplus \exists R \ar[dl]\\
 & S ^{ 1 } \ar@{-->}[ul] & & S ^{ 2 } \ar@{-->}[ul] & \cdots & 
 & & S ^{ n } \ar@{-->}[ul]}
\end{align}
such that
\begin{itemize}
\item
the lower triangles are distinguished triangles,

\item
each
$
 S ^{ i }
$
is a direct sum of shifts of $T$.
\end{itemize}
By applying the projection functor
$
 p _{ \cA }
$
to this diagram, we immediately see that
$
 a \in \cA
$
is classically generated by
$
 p _{ \cA } \lb T \rb
$.
\end{proof}

\begin{lemma}\label{lm:semi-continuity}
Let $X$ be a projective scheme,
$
 a \in \bfD ^{ b } \lb \coh X \rb
$
a bounded complex of coherent sheaves, and
$
 b \in \bfD ^{ \perf } \lb X \rb
$
a perfect complex. Let $T$ be a scheme
of finite type over $\bfk$ with a point $0 \in T(\bfk)$, and
$M$ a line bundle on $X\times _{ \bfk }T$ such that
$
 \RHom(b,a\otimes M_0)=0.
$
Then there exists an open neighborhood
$
 0 \in U \subset T
$
such that for any $t\in U(\bfk)$
$
 \RHom ( b, a \otimes M_t ) = 0.
$
\end{lemma}

\begin{proof}
It follows from the assumptions on
$
 a
$
and
$
 b
$
that
$
 \bR\cHom (p_X^*b, p_X^*a \otimes M))
$
is a bounded complex of coherent sheaves on
$
 X \times _{ k } T
$.
Since
$
 p _{ T }
$
is projective, it then follows that
\begin{equation}
 S = \Supp { ( \bR p_{T*}\bR\cHom (p_X^*b, p_X^*a \otimes M))} \subset T
\end{equation}
is a closed subset of
$
 T
$.

Next consider the sequence of isomorphisms
\begin{equation}
 \begin{split}
 \bL \iota_t^* \bR p_{T*} \bR\cHom ( p_X^*b, p_X^*a \otimes M)
 \simto \bL \iota_t^* \bR p_{T*} \lb p_X^* \bR\cHom ( b, a ) \otimes M \rb\\
 \simto \bR \Gamma \lb X, \bR\cHom ( b, a ) \otimes M_t \rb
 \simto \RHom ( b, a \otimes M_t ),
\end{split}
\end{equation}
where
$
 \iota_t \colon \{ t \} \hto T
$ is the natural inclusion. The second isomorphism follows from the base change theorem for
flat morphisms (\cite[Corollary 2.23]{MR2238172}).
From this and by \pref{lm:Derived_NAK}, we see that the subset $S$ does not contain $0$. 
Now we can define $U$ as the complement
of $S$.
\end{proof}

\begin{lemma}\label{lm:open}
Let $X$ be a projective scheme over a field
$
 \bfk
$,
and 
$
 \bfD ^{ \perf } \lb  X \rb
 =
 \la \cA, \cB \ra
 = 
 \la \cA', \cB' \ra
$
be two SODs. 
Let $T$ and $M$ be as in \pref{lm:semi-continuity}. 
Then the subset 
$
 U(\cA')
 =
 \lc t \in T ( \bfk ) \mid \cA \otimes M_t = \cA' \rc
\subset T(\bfk)
$
is open. 
\end{lemma}

\begin{proof}
The assertion is trivial when
$
 U \lb \cA ' \rb = \emptyset
$,
we assume 
$
 U \lb \cA ' \rb \neq \emptyset
$.
It is enough to check the following two claims separately
for each point $0\in\ U(\cA')$.
\begin{enumerate}
\item There exists an open neighborhood
$
 0 \in U
$ such that
$
 \cA \otimes M_t\subset\cA'
$ 
holds for any 
$
 t \in U
$. 
\item There exists an open neighborhood
$
 0 \in U
$ such that
$
 \cA \otimes M_t^{-1}\subset\cA'
$ 
holds for any 
$
 t \in U
$. 
\end{enumerate}
We give a proof only for the first one; the second follows from this by replacing
$
 M
$
with
$
 M ^{ - 1 }
$.

By \pref{lm:classical_generator}, we can take classical generators
$
 a \in \cA
$
and
$
 b ' \in \cB '
$.
Then we have the useful criterion
\begin{equation}
 \cA \otimes M_t \subset \cA' = \cB'^{\perp}
 \iff
 \RHom ( b', a \otimes M_t )
 = 0.
\end{equation}
Since the latter condition on
$
 t
$
is known to be open by \pref{lm:semi-continuity}, we are done.
\end{proof}

\begin{proof}[Proof of \pref{th:rigidity}]
By \pref{th:characterization_of_points_of_Pic0} and
\pref{df:algebraic_equivalence}, it is enough to show the following

\begin{claim}
Under the same assumptions as in \pref{th:rigidity},
let $T$ be a connected scheme of finite type over $\bfk$, and
$M$ a line bundle on $X\times_{\bfk}T$.
If $\cA\otimes M_0=\cA$ holds for $0\in T(\bfk)$,
then $\cA\otimes M_t=\cA$ holds for all $t\in T(\bfk)$.
\end{claim}

In order to show the claim, set
$
 S
 =
 \{ \cA \otimes M_t | \ t \in T ( \bfk ) \}. 
$ 
 By \pref{lm:open}, we obtain a decomposition
$
 T(\bfk)=\coprod_{\cA'\in S}U(\cA')
$
of $T(\bfk)$ into disjoint open subsets. Since $T(\bfk)$ is connected by the
assumption $\bfk=\bfkbar$,
this implies that $U(\cA)=T(\bfk)$.
\end{proof}

Let $X$ be a smooth projective variety over $\bfk$.
Using similar arguments, we can also show
$
 g^*\cA = \cA \subset \bfD ^{ b } \lb \coh X \rb
$
for any automorphism
$
 g \in \AAut^0 _{ X / \bfk }
$.
Here
$
 \AAut^0 _{ X / \bfk }
$
is the identity component of the group scheme
$
 \AAut _{ X / \bfk }
$
(see \cite[p.133 Exercise]{MR2222646} for the definition
and the existence of $\AAut_{X/\bfk}$).
Thus we obtain
\begin{corollary}\label{cr:rigidity_with_repsect_to_toplologically_trivial_autoequivalences}
For any SOD
$
 \bfD ^{ b } \lb \coh X \rb = \la \cA, \cB \ra
$
and
$
 \Phi \in \PPic^0 _{ X / \bfk } \rtimes \AAut^0 _{ X / \bfk }
$,
we have
$
 \Phi \cA = \cA \subset \bfD ^{ b } \lb \coh X \rb.
$
\end{corollary}

\noindent
As proven in \cite[Theorem 2.12]{rosay2009some}, the group scheme
$
 \PPic^0 _{ X / \bfk } \rtimes \AAut^0 _{ X / \bfk }
$
coincides with the identity component of the group scheme of autoequivalences of $\bfD ^{ b } \lb \coh X \rb$.

%
%

\section{More results for surfaces}\label{sc:Surfaces}

In this section we closely investigate the non-existence of SOD for minimal projective surfaces.
For $\kappa = 0$, we completely understand when and only when there is a non-trivial SOD.
For $\kappa = 1$, we have a rather satisfactory answer that there is no SOD if
$
 p _{ g } > 0
$.
For $\kappa = 2$, we show the non-existence of SOD under strong assumptions. The assumption seems to be a technical one, and it remains to be a future task to get rid of it.

Readers can refer to \cite{MR986969} for various notions for surfaces in positive characteristics,
such as non-classical Enriques surfaces, quasi-elliptic fibrations, wild fibers, and quasi-bielliptic surfaces. The following result plays a central role for the cases
$
 \kappa = 0, 1
$.

Let $f\colon X\to C$ be a relatively minimal (quasi-)elliptic surface
with multiple fibers
$
 X _{ c_i } = m_i F_i \quad (i = 1, \dots, k)
$.
Then we have the Kodaira-Bombieri-Mumford canonical bundle formula (see \cite[Theorem 2]{MR0491719})
\begin{equation}\label{eq:Kodaira}
\omega_X \simeq f^*(\omega_C\otimes \bR ^1f_*\cO_X/T)\otimes
\cO_X
\lb \sum_i a_i F_i
\rb,
\end{equation}
where
$
 T \subset
 \bR ^1 f_* \cO_X
$
is the torsion part and
$
 0 <  a_i \le m_i - 1
$
are some integers.
It is known that
$
 \omega_C \otimes \bR ^1 f_*\cO_X / T
$
is a line bundle of degree
$
 2g(C) - 2 + \chi ( \cO_X ) + \mathop {\textrm{length}} { T }
$.

%
%

\subsection{$\kappa=0$}

Since classical Enriques surface satisfies $p_g=q=0$, any line bundle on it
is exceptional.
Hence the derived category always admits a non-trivial SOD
(see \cite{MR3302614} and \cite{hosono2015derived}
 for further results on this topic).
For non-classical Enriques, abelian, and K3 surfaces we have no SOD by \pref{cr:iff},
since their canonical bundles are trivial.

The most non-trivial is the following
\begin{proposition}
Let
$
 X
$
be an (quasi-)bielliptic surface. Then
$
 \bfD ^{ b } \lb \coh X \rb
$
admits no SOD.
\end{proposition}

\begin{proof}
The idea is to use
\pref{cr:trivial_chern_class}.
In order to check its assumption, take a (quasi-)elliptic fibration
$
 f \colon X \to C
$
without multiple fibers.
By \eqref{eq:Kodaira}, there exists a line bundle $L$ on $C$
such that
$
 \omega_X \simeq f^* L
$.
Since
$
 f^* \colon \PPic _{ C / \bfk } \to \PPic _{ X / \bfk }
$
preserves the identity components, it is enough to show that
$
 [L] \in \PPic^{0} _{ C / \bfk }
$.

By the assumption there is a positive integer $N$ for which
$
 \omega_X ^{ \otimes N } \simeq \cO _{ X }
$.
Together with the projection formula, this implies
\begin{align}
 L ^{ \otimes N }
 \simeq
 f _{ * } f ^{ * } L ^{ \otimes N }
 \simeq
 f _{ * } \cO _{ X }
 \simeq
 \cO _{ C }.
\end{align}
Since $C$ is a smooth projective curve, this implies
$
 [L] \in \PPic^{0} _{ C / \bfk }
$.
Thus we conclude the proof.
\end{proof}

%
%

\subsection{$\kappa=1$}\label{sc:kappa=1}

\begin{theorem}\label{th:elliptic_surface}
Let $f\colon X\to C$ be a relatively minimal elliptic or quasi-elliptic surface.
If $p_g(X)>0$, then $\bfD ^{ b } \lb \coh X \rb$ has no SOD.
In particular, this applies to the case when $X$ is a minimal surface of Kodaira dimension $1$ with $p_g ( X ) >0$.
\end{theorem}

\begin{proof}
Since the contribution from the multiple fibers in the RHS
of \pref{eq:Kodaira} is fixed as a linear system and $f$ is an algebraic fiber space,
if $p_g(X)>0$, then $\Bs| \omega_X |$ is a union of finitely many fibers of
$
 f
$.
The finite set
$
 f ( \Bs | \omega _X | ) \subset C
$
will be denoted by
$
 S
$.

Take any SOD
$
 \bfD ^{ b } \lb \coh X \rb = \la \cA, \cB \ra
$.
By \pref{th:pg positive}, either
$
 \cA
$
or
$
 \cB
$
is supported in
$
 f ^{ - 1 } ( S )
$.
This implies that the SOD under consideration is
$
 C
$-linear in the sense of \cite{MR2801403}; to see this,
note that the pull-back of any locally free sheaf on $C$
is trivial on an open neighborhood of
$f ^{ - 1 } ( S )$. Also since $f$ is flat, we can apply
\cite[Theorem 5.6]{MR2801403}
to any base change of
$
 f
$.
In particular, for any closed point
$
 s \in S
$
we obtain the SOD
$
 \bfD ^{ \perf } \lb X _{ s } \rb = \la \cA _{ X_s } ^{ \perf }, \cB _{ X_s } ^{ \perf } \ra
$
(following the notation of \cite{MR2801403}).
On the other hand, since
$\omega_{ X _s }$ is torsion and any torsion line bundle on
a complete curve over a field is contained in its
$
 \PPic^0
$
(see \cite[Chapter 0 Section 7]{MR986969}),
we can use \pref{th:rigidity} to see
$
 \cA _{ X_s }^{ \perf } \otimes \omega _{ X_s } = \cA _{ X_s } ^{ \perf } \subset \bfD ^{ \perf } \lb X_s \rb
$.
Since the Serre duality works in
$
 \bfD ^{ \perf } \lb X_s \rb
$
(see, say, \cite[Proposition 7.47 and Remark 7.48]{MR2434186}), this implies that the SOD is actually an OD.

Now since there is no OD of
$
 \bfD ^{ \perf } \lb X _{ s } \rb
$
 by \pref{lm:no_OD}, it follows that either
$
 \cA _{ X_s } ^{ \perf } = 0
$
or
$
 \cB _{ X_s } ^{ \perf } = 0
$
should hold for any
$
 s \in S
$.

Without loss of generality, let us assume that
$
 \cA
$
is supported in
$
 f ^{ - 1 } ( S )
$
for the rest of proof.
Assume for a contradiction
$
 \cA \neq 0
$.
By the construction of
$
 \cA _{ X_s } ^{ \perf }
$
given in \cite[Section 5.4]{MR2801403},
we see that for any object
$
 a \in \cA,
$
its base change
$
 a_s = a {}^{ \bL }\!\!\otimes _{ \cO _C } \cO _s
$
is contained in
$
 \cA _{ X_s } ^{ \perf }
$
(note that
$
 \bfD ^b \lb \coh X \rb  = \bfD ^{ \perf } \lb X \rb
$).
Therefore, by \pref{lm:Derived_NAK},
there should be a point
$
 s \in S
$
for which
$
 \cA _{ X_s } \neq 0
$.
Then we obtain
$
 \cB _{ X_s } = 0
$,
so that any object
$
 b \in \cB
$
satisfies
$
 \Supp b \cap f ^{ - 1 } ( s ) = \emptyset
$
by
\pref{lm:Derived_NAK} again.
Since
$
 f
$
is proper, this implies that any object
$
 b \in \cB
$
is supported in a union of finitely many fibers.
Thus we conclude that any object in either
$
 \cA
$
or
$
 \cB
$
should be supported in a union of finitely many fibers, and it clearly contradicts the assumption that
$
 \bfD ^{ b } \lb \coh X \rb
$
is generated by
$
 \cA
$
and
$
 \cB
$.
\end{proof}

\begin{remark}
\begin{enumerate}
\item
Our method is also applicable to other situations in which
we have a sufficiently nice canonical bundle formula (see \pref{sc:Toward_higher_dimensions}).

\item
(Assume $\bfk = \bC$ for simplicity.)
Let $X$ be a minimal projective surface with $\kappa{(X)}=1$.
If $p_g(X)=0$, $h^1(\cO_X)$ should be either $0$ or $1$ since
$\chi{(\cO_X)}\ge 0$ holds (see \cite[Chapter V, \S 12]{Barth-Hulek-Peters-Van_de_Ven}).
If $h^1(\cO_X)=0$, as we saw before, any line bundle is exceptional.
If $h^1(\cO_X)=1$ (and hence $\chi{(\cO_X)}=0$), although we can restrict the nature
of the fibration as follows, we do not know if $\bfD ^{ b } \lb \coh X \rb$ can admit an SOD or not.
\begin{itemize}
\item $g(C)=0$ or $1$ by the canonical bundle formula and the assumption $p_g(X)=0$.
\item Smooth fibers are all isomorphic to one another and the multiple fibers are of type $_mI_0$ for some
$m>0$ (see \cite[Chapter III, \S 18]{Barth-Hulek-Peters-Van_de_Ven}).
\end{itemize}
\end{enumerate}
\end{remark}

%
%

\subsection{$\kappa=2$}
We apply the results of \pref{sc:Local_situation} to
smooth projective minimal surfaces with
$
 \kappa = 2
$
(i.e., of general type).
There are some examples of minimal surfaces of general type on which the
connected components of fixed part of the canonical linear system can be birationally
contracted to points (in the category of algebraic spaces).
This property turns out to ensure the non-existence of SOD.

\begin{theorem}\label{th:negative_definite}
Let $X$ be a smooth projective minimal surface of general type with $p_g\ge 2$.
Assume the following condition (*).
\begin{quote}
(*)
For any one-dimensional connected component $Z$ of $\Bs| \omega_X |$,
its intersection matrix is negative definite.
\end{quote}

Then $\bfD ^{ b } \lb \coh X \rb$ has no SOD.
\end{theorem}

\begin{proof}
By \pref{cr:induction},
it is enough to show that for any one-dimensional connected component $Z$ of
$\Bs| \omega_X |$, the category
$
 \cC_Z
$
has no SOD.
We prove this for more general
$
 \cC_W
$,
where
$
 W
$
is any reduced connected one cycle contained in $\Bs| \omega_X |$,
by an induction on the number of irreducible components of $W$.
If $W$ is empty, there is nothing to show.
In general we can use the following

\begin{lemma}\label{lm:non-vanishing}
Under the assumptions of \pref{th:negative_definite},
let $W$ be a reduced connected one-cycle which is contained in $\Bs| \omega_X |$.
Then there exists
$
 s = ( s_m ) _{ m \ge 1 } \in \varprojlim H^0 ( m W, \omega _{ X } | _{ m W } )
$
which does not vanish at the generic point of an irreducible component of $W$.
\end{lemma}

\noindent
This implies the strict inequality
$
 B \subsetneq W
$
(see \pref{cr:induction} for notation).
Since $W$ has pure dimension $1$, we can pick an irreducible curve
$
 Z_1 \subset W
$
which is not contained in $B$.
By \pref{cr:induction} it is then enough to show the non-existence of SOD
for $\cC_{W'}$, where
$
 W'
$
are connected components of
$
 \overline{ W \setminus Z_1 } \cup ( B \cap Z_1 ).
$
Since $W'$ is either a point or a connected one cycle whose number of irreducible
components is strictly less than that of $W$, we can apply the induction hypothesis.
\end{proof}

\begin{proof}[Proof of \pref{lm:non-vanishing}]
By the Riemann-Roch
(\cite[Chapter II, Theorem 3.1.]{Barth-Hulek-Peters-Van_de_Ven})
and the adjunction formula, we see
$
 h^0 ( \omega _X | _{W} ) = h^1(\omega_X|_{W}) + \frac{1}{2} W \cdot ( K_X - W ).
$
Since we assumed
$
 p_g ( X ) > 1,
$
$
 K_X - W
$
is linearly equivalent to a non-zero effective divisor.
Hence by the $2$-connectedness of the canonical divisor of minimal surfaces of general type
(\cite[Chapter VII, Proposition 6.2. (ii)]{Barth-Hulek-Peters-Van_de_Ven}), we see
$
 W \cdot (K_X - W) \ge 2
$. Thus we obtain $h^0(\omega_X|_{W})>0$.

Take any non-zero global section $s_1$ of $\omega_X|_{W}$.
Since $W$ is reduced, $s_1$ is generically non-vanishing on at least one
irreducible component of $W$.
For each $m>1$ we show that the global section $s_{m-1}$ of
$
 \omega _X | _{ ( m - 1 ) Z }
$
lifts to a global section
$s_m$ of
$
 \omega_X | _{ mW }
$, so as to obtain the desired
$
 s = ( s_m ) _{ m \ge 1 } \in \varprojlim H^0 ( m W, \omega _{ X } | _{ m W } )
$.
Consider the exact sequence
\begin{equation}
 0
 \to
 \cO_{W}( K_X - ( m - 1 ) W )
 \to
 \cO_{mW} ( K_X )
 \to
 \cO _{ ( m - 1 ) W } ( K_X )
 \to 0
\end{equation}
and the associated cohomology long exact sequence.
This yields an exact sequence
\begin{equation}
 H^0(mW, \cO_{mW}(K_X))\to H^0((m-1)W, \cO_{(m-1)W}(K_X))
 \to
 H^1(W, \cO_{W}(K_X-(m-1)W)),
\end{equation}
and hence it is enough to show the vanishing of the third term.
By the adjunction formula and the Serre duality
for embedded curves (see \cite[Chapter II, Section 1]{Barth-Hulek-Peters-Van_de_Ven}),
its dimension can be rewritten as
$
 h^1(W, \cO_{W}(K_X-(m-1)W))= h^0(W, \cO_W(mW))
$.
Finally, the vanishing of the RHS follows from the assumption $W^2<0$.
\end{proof}

\begin{example}\label{eg:Horikawa}
Minimal surfaces $X$ of general type with $p_g=K_X^2=2$ and
$q=0$ were investigated in \cite{MR517773}. Among them, those of type III
(see \cite[page 104]{MR517773}) satisfy the assumption of \pref{th:negative_definite}.
In fact the fixed part consists of a ($-2$)-curve. 
In this example the moving part
of the canonical linear system is base point free and defines a genus two pencil over the projective line
(\cite[Theorem 1.3]{MR517773}).
\end{example}
%

%
%

\section{$\kappa = 1$ in higher dimensions}\label{sc:Toward_higher_dimensions}

Most of the arguments of \pref{sc:kappa=1} can be generalized
to higher dimensions, except one point.

\begin{theorem}\label{th:when kappa=1}
Let $X$ be a non-singular projective $n$-fold defined over $\bfk$ such that
$X$ is a minimal model with $\kappa{(X)}=1$ and $p_g(X)>0$.
Suppose
$
 \omega _X
$
is semi-ample so that the canonical morphism
$
 f \colon X \to C
$
exists.
Suppose that for any scheme-theoretic fiber $X_c$ of $f$ we have
$
 \omega _{ X } | _{ X_c } \in \PPic^0 _{ X_c }.
$
Then $\bfD ^{ b } \lb \coh X \rb$ admits no SOD.
\end{theorem}
\begin{proof}
Below we prove that
$
 \Bs | \omega_X |
$
is contained in the union of finitely many fibers of $f$.
Once it is shown, with the extra assumption
$
 \omega_X|_{X_c} \in \PPic^{0} _{X_c/\bfk}
$,
which was automatically fulfilled in the case of \pref{sc:kappa=1},
the arguments of \pref{sc:kappa=1} works without change.

The assumption $p_g(X)>0$ implies that
$
 p_g ( \omega _{ X_c } ) > 0
$ holds for a general fiber $X_c$.
Since $X$ is irreducible and $C$ is a non-singular curve, the morphism $f$ is flat
(\cite[Chapter III, Proposition 9.7]{Hartshorne}). Combined with the torsion-freeness
\cite[Theorem 2.1]{MR825838} and the theory of cohomology and base change
\cite[Chapter III,
Theorem 12.11]{Hartshorne}, we see that the direct image $f_*\omega_X$ is an
invertible sheaf.
The natural injective morphism
$
 f^* f_* \omega_X \to \omega_X
$
provides us with an effective divisor $E$ on $X$
which fits in the canonical bundle formula
\begin{equation}
 \omega_X\simeq f^*f_*\omega_X \otimes \cO _X(E).
\end{equation}
This formula is a generalization of \eqref{eq:Kodaira}.

Arguing as in \cite[Proof of Theorem 12.1, Chapter V]{Barth-Hulek-Peters-Van_de_Ven}, we see that
the morphism $f^*f_*\omega_X \to \omega_X$ is an isomorphism on smooth fibers:
in fact, for a smooth fiber $X_c$
we can find a section $\stilde\in (f_*\omega_X)_c$ which,
under the isomorphism
$
 f_*\omega_X\otimes_{\cO_C}\bfk(c)
 \simto
 H^0(X_c,\omega_X|_{X_c}),
$
corresponds to the global trivialization of $\omega_X|_{X_c} \simeq \omega_{X_c}$.
Since $f$ is projective, there exists an open neighborhood $U\ni c$ such that $\stilde$ is
well-defined and vanishes nowhere on $f^{-1}(U)$. This shows that the morphism
$f^*f_*\omega_X\to\omega_X$ is surjective (and hence is an isomorphism) on $f^{-1}(U)$.
Thus we see that $E$ is contained in the union of the singular fibers of $f$.

Now since
\begin{align}
 H ^{ 0 } \lb X, f ^{ * } f _{ * } \omega _{ X } \rb
 \simeq
 H ^{ 0 } \lb C, f _{ * } \omega _{ X } \rb
 \simeq
 H ^{ 0 } \lb X, \omega _{ X } \rb \neq 0
\end{align}
by the assumption, it follows from the sequence of maps
\begin{align}
 H ^{ 0 } \lb C, f _{ * } \omega _{ X } \rb
 \xto[\simeq]{f ^{ * }}
 H ^{ 0 } \lb X, f ^{ * } f _{ * } \omega _{ X } \rb
 \xto[]{ \otimes s _{ E } }
 H ^{ 0 } \lb X, \omega _{ X } \rb,
\end{align}
where
$
 s _{ E }
$
is the section of the line bundle
$
 \cO _{ X } ( E )
$
which tautologically corresponds to $E$,
that there exists an effective canonical divisor of $X$ whose support is contained in the union of finitely many fibers and the support of $E$.
Thus we conclude the proof.
\end{proof}

Let
$
 \PPic^{\tau} \subset \PPic
$
be the subscheme of numerically trivial line bundles
(see \cite[Section 9.6]{MR2222646}).
If
$
 \PPic^0_{X_c}=\PPic^{\tau}_{X_c}
$
holds for any singular fiber
$
 X_c
$,
since the morphism $f$ is defined by some multiple of $\omega_X$,
the last assumption of
\pref{th:when kappa=1} will be automatically satisfied.
This is always the case if $\dim X_c = 1$, but not in general.
Actually, even worse, the following example
due to Keiji Oguiso satisfies all the assumptions of \pref{th:when kappa=1}
but the last one. The authors are not sure if its derived category
admits a non-trivial SOD or not.

\begin{example}
In this example we assume
$
 \bfk = \bC
$
for simplicity.
Fix an integer $n\ge 4$ such that $n+1$ is a prime number.
We construct a minimal $n$-fold $X$ of $\kappa{(X)}=1$ and $p_g(X)>0$ such that
the canonical morphism
$
 f \colon X \to C
$
has a singular fiber
$
 X_c 
$
which satisfies
$
 \omega _{ X } | _{ X_c } \nin \PPic^0 _{ X_c }
$.

Consider the Fermat hypersurface
 $ Y = \lb \sum _{ i = 0 } ^{ n } X _i ^{ n + 1 } = 0 \rb \subset \bP ^n$.
It is a smooth projective Calabi-Yau ($n-1$)-fold in the strict sense;
i.e., 
$
 H ^{ i } ( \cO_Y )
 =
 0
$
for
$
 i = 1, 2, \dots, n - 2
$
and
$
 \omega_Y \simeq \cO_Y
$.
Note that one can describe the unique (up to constant) global section of
$
 \omega _{ Y }
$
as the residue top form
\begin{equation}
 \psi
=
\Res_{Y}
\frac
{\sum_{i=0}^{n} (-1)^i X_i d X_0 \wedge \cdots \wedge \widehat{d X_i} \wedge \cdots \wedge dX_n}
{\sum_iX_i^{n+1}}.
\end{equation}

Pick a primitive
$
 ( n + 1 )
$th root of unity
$
 \zeta
$
to define the action of the cyclic group
$
 G = \bZ / ( n + 1 ) \bZ
$
on $Y$ given as
$
 X_i \mapsto \zeta^i X_i
$.
Since we assumed that $ n+1 $ is a prime number, this action is free and hence
we obtain the non-singular quotient
$
 Z = Y / G
$.
Since the top form $\psi$ is easily seen to be
$
 G
$-invariant, it follows that
$Z$ is also a Calabi-Yau $(n-1)$-fold.

Let
$
 C'
$
be the smooth projective model of the affine curve
$
 ( y^2-x^{4(n+1)}+1 = 0 ) \subset \bA ^2,
$
and let $G$ act on $C'$ by 
$
 ( x, y ) \mapsto ( \zeta x, y )
$.
This action is effective but not free.
As can easily be seen,
$
 g ( C' ) = 2 ( n + 1 )
$
and
$
 \gamma = ( y ^{ - 1 } x ^n ) d x
$
defines an $G$-invariant regular $1$-form on $C'$.

Now consider the \'etale quotient
$
 \pi \colon Y \times C' \to ( Y \times C' ) / G =: X
$.
Since
$
 \psi \boxtimes \gamma \in
 H^0 ( Y \times C', \omega_{Y\times C'})^G \simeq H^0 ( X, \omega_X ),
$
we see $p_g(X)>0$.
Also since
$
 \omega _{ Y \times C' } \simeq \pi^* \omega _X
$,
$X$ is minimal and
$\omega_X$ is not trivial. Combined with the inequalities
$
 0 \le \kappa { ( X ) } \le \kappa { ( Y \times C' ) }
 = 1
$,
we see
$
 \kappa { ( X ) }
 = 1
$.
Also it is easily seen that the algebraic fiber space
$
 f \colon X \to C'/G = \colon C
$
is induced by the pluri-canonical linear system of
$
 X
$.

Let $c\in C$ be a branch point of $C'\to C$.
By an abuse of notation we write
$
 (X_c) _{\mathrm{red}} = Z
$, so that
$
 X_c = (n+1) Z
$.
Here we claim that
$
 \cO_{X_c} ( K_X ) \nin \PPic^0_{X_c/\bfk}
$.
To see this we prove
$
 \cO_Z ( K_X ) \nin \PPic^0_{Z/\bfk}
$,
which in turn is equivalent to
$
 \cO_Z ( Z ) \not \simeq \cO_Z
$.

Take a Stein open neighborhood
$
 c \in V \subset C
$
such that
$
 X_c = (n+1)Z
$
is a deformation retract of
$
 U  =f^{-1} ( V ) \subset X
$
(see
\cite[Chapter I, Theorem 8.8]{Barth-Hulek-Peters-Van_de_Ven}).
Consider the following commutative diagram with exact rows.

\begin{equation}\label{eq:diagram}
 \xymatrix{
H^1(U,\cO_U) \ar[r] \ar[d] & H^1(U,\cO^{*}_U) \ar[r]^{c_1} \ar[d] &
H^2(U,\bZ) \ar[r] \ar[d]_{\simeq} & H^2(U,\cO_U) \ar[d] \\
H^1(X_c,\cO_{X_c}) \ar[r] \ar[d] & H^1(X_c,\cO^*_{X_c}) \ar[r]^{c_1} \ar[d] &
H^2(X_c,\bZ) \ar[r] \ar[d]_{\simeq} & H^2(X_c,\cO_{X_c}) \ar[d] \\
H^1(Z,\cO_Z) \ar[r]  & H^1(Z,\cO^*_Z) \ar[r]^{c_1} &
H^2(Z,\bZ) \ar[r] & H^2(Z,\cO_Z)\\}
\end{equation}

We check that the terms in the 1st and the 4th columns in the diagram
\eqref{eq:diagram} all vanish.
First, since $Z$ is a Calabi-Yau $(n-1)$-fold with $n\ge 4$, $H^i(Z,\cO_Z)=0$ for
$i=1,2$.
For the remaining four terms, note first that the higher direct image sheaves $\bR ^if_*\cO_X$
are locally free for all $i\ge 0$ due to the torsion-freeness
theorem \cite[Theorem 2.1]{MR825838}, relative duality, and the fact that the base space $C$ is a
non-singular curve.
From this we obtain the isomorphisms
$
 \bR ^i f_* \cO_X \otimes \bC ( t ) \simeq H^i ( X_t, \cO _{ X_t } )
$
(\cite[Chapter III, Theorem 12.11]{Hartshorne}).
Since $H^i(X_t,\cO_{X_t})=0$ for $i=1,2$ and general $t$,
we obtain $\bR ^if_*\cO_X=0$ for $i=1,2$ and hence the desired vanishings.

As a result, it turns out that
the six terms in the 2nd and the 3rd columns in \eqref{eq:diagram}
are isomorphic to one another.
Since we can easily check that $\cO_U(Z)\in H^1(U,\cO^*_U)$ is a $(n+1)$-torsion non-trivial
line bundle, so is $\cO_{Z}(Z)$.
\end{example}
\begin{remark}
The structure sheaf $\cO_Z$ is \emph{not} an exceptional object.
Actually one can easily see from the calculations above that
$
 \Ext^{n-1}_X ( \cO_Z, \cO_Z ) \neq 0
$.
\end{remark}

%
%

\section{Generalization to twisted sheaves}
\label{sc:twisted_sheaves}

Most of the results we have established so far can be generalized to
derived categories of \emph{twisted} coherent sheaves without essential change.

\begin{definition}
A \emph{cohomological Brauer class} of a scheme
$
 X
$
is an element
$
 \alpha \in
 \Br' ( X ) :=
 H^2 _{ et } ( X, \cO_X^* )
$.
A pair
$
 ( X, \alpha )
$
will be called a
\emph{cohomological Brauer pair}.
\end{definition}

Given such a pair
$
 ( X, \alpha )
$,
we can define the abelian category
$
 \coh ( X, \alpha )
$
of
$
 \alpha
$-twisted coherent sheaves.
When
$
 X
$
is defined over a field
$
 \bfk
$,
it comes with the structure of a
$
 \bfk
$-linear category.

Fix an \'etale cover
$
 \cU = ( U_i ) _{ i \in I }
$
of
$
 X
$
on which the cohomology class
$
 \alpha
$
is represented by a \v{C}ech cocycle
\begin{equation}
 \alpha = ( \alpha _{ i j k } ) _{ i, j, k \in I }
 \in
 \Zv^2 ( \cU, \cO ^{ * } )
 =
 \prod _{ i, j, k \in I } H^0 ( U _{ i j k }, \cO ^{ * } ),
\end{equation}
where
$
 U _{ i j k }
 =
 U_i \times _{ X } U_j \times _{ X } U_k
$
(by an abuse of notation,
we used the same symbol
$
 \alpha
$
to describe its representative).
Then an
$
 \alpha
$-twisted coherent sheaf
$
 F
$
is a collection of coherent sheaves
$
 F_i \in \coh U_i
$
and isomorphisms
$
 \varphi _{ i j } \colon F_j | _{ U _{ i j } } \simto F_i | _{ U _{ i j } }
$
which satisfy the
$
 \alpha
$-twisted cocycle conditions
\begin{equation}
 \varphi _{ i j } \varphi _{ j k } \varphi _{ k i } = \alpha _{ i j k } \cdot \id _{ U _{ i j k } }
 \colon
 F _{ i } |_{ U _{ i j k } } \simto F _{ i } |_{ U _{ i j k } }.
\end{equation}
A morphism between such data is a collection of
$
 \cO _{ U_i }
$-homomorphisms
which satisfy the obvious consistency.
Then we can check that thus obtained category is abelian
and is independent of the choice of a representative of
$
 \alpha
$
(see \cite[Lemma 1.2.3]{caldararu2000derived}).
We write
$
 \bfD ( X, \alpha ) = \bfD ^b \coh ( X, \alpha )
$,
so that
$
 \bfD ( X, 0 ) = \bfD ^{ b } \lb \coh X \rb
$.

\begin{definition}\label{df:support_for_twisted_sheaves}
For an
$
 \alpha
$-twisted coherent sheaf
$
 F \in \coh ( X, \alpha )
$,
its support
$
 \Supp F
$
is defined as the closed subscheme
$
 \Spec _{ X } \lb \Image ( \cO_X \to \cEnd ( F ) ) \rb
 \subset
 X
$.
For an object
$
 F \in \bfD ( X, \alpha )
$,
its support is defined as
$
 \Supp F = \bigcup _{ i } \Supp \cH^i ( F ) _{ \mathrm{red} }
$.
\end{definition}

For
$
 \alpha
$-twisted sheaves
$
 F
$
and
$
 G
$,
we can define the (untwisted!) coherent sheaf of homomorphisms
$
 \cHom ( F, G ) \in \coh ( X )
$.
The following fact is an easy consequence of this observation.

\begin{lemma}\label{lm:Serre_functor_for_twisted_sheaves}
Let
$
 ( X, \alpha )
$
be a smooth proper cohomological Brauer pair. Then
$
 \otimes \omega _{ X } [ \dim X ]
$
is the Serre functor of
$
 \bfD ( X, \alpha )
$.
\end{lemma}
\begin{proof}
See \cite[Example 1.4.3]{navas2010fourier}
\end{proof}

The next lemma is a direct consequence of the definition of
twisted coherent sheaves.

\begin{lemma}\label{lm:point_module}
For any closed point
$
 x \in X
$,
its structure sheaf
$
 \bfk ( x )
$
gives rise to the
\emph{``$\alpha$-twisted skyscraper sheaf at
$
 x
$''}
$
 \in \bfD ( X, \alpha )
$
for any cohomological Brauer class
$
 \alpha
$.
\end{lemma}

It is sometimes convenient to restrict ourselves to \emph{Brauer classes}.
Brauer classes form a subgroup
$
 \Br ( X ) \subset \Br' ( X )
$,
and they are characterized by either of the following properties
(see \cite[Theorem 1.3.5]{caldararu2000derived}).

\begin{itemize}
 \item There exists a sheaf of Azumaya algebras which represents the class $\alpha$.
 \item There exists a non-zero locally free $\alpha$-twisted sheaf of finite rank.
\end{itemize}

\noindent
The difference of these two notions are very subtle.
In fact, it is shown in \cite[Theorem 1.1]{de2003result} that
$
 \Br = \Br'
$
holds on any projective scheme.

Now that we have prepared basic lemmas,
we can show

\begin{theorem}
Let
$
 ( X, \alpha )
$
be a smooth proper Brauer pair.
Assume that
$
 \omega_X
$
is trivial on an open neighborhood of the canonical base locus
$
 \Bs | K_X |
$.
Then
$
 \bfD ( X, \alpha )
$
admits no SOD.
\end{theorem}
\begin{proof}
Since we assumed
$
 X
$
is proper, it satisfies the resolution property by
\cite[Theorem 1.2]{MR2108211}.
Combined with the assumption
$
 \alpha \in \Br ( X )
$,
\cite[Lemma 2.1.4]{caldararu2000derived}
implies that any coherent
$
 \alpha
$-twisted sheaf on
$
 X
$
is a quotient of a locally free
$
 \alpha
$-twisted sheaf of finite rank.
Hence those sheaves form a spanning class of
$
 \bfD ( X, \alpha )
$,
and the original proof of \pref{lm:no_OD} can be used without change
to show the non-existence of OD.
Then the original proof of
\pref{th:locally free} works with a minor modification,
by replacing the corresponding lemmas with
\pref{lm:Serre_functor_for_twisted_sheaves}
and
\pref{lm:point_module}.
In order to show that both of the SOD summands can not contain
closed points at the same time, one can use a locally free twisted sheaf instead of
$
 \cO_X
$.
\end{proof}

Similarly, by using \cite[Lemma 2.1.4]{caldararu2000derived}, the original proof of
\pref{th:negative_definite}
works without change and we obtain

\begin{theorem}
Let
$
 ( X, \alpha )
$
be a smooth projective Brauer pair such that
$
 X
$
is a minimal surface of general type satisfying
$
\dim_{\bfk} H^0(X, \omega_X) > 1
$
and the condition
(*).
Then
$
 \bfD ( X, \alpha )
$
admits no SOD.
\end{theorem}

\begin{remark}
We expect that the other results can be generalized as well with a bit more effort.
For \pref{th:rigidity}, all we need is the
existence of classical generators in the derived category of
twisted coherent sheaves.
Similarly, for \pref{th:when kappa=1},
we have to check that the base change theorem
\cite[Theorem 5.6]{MR2801403} works for Brauer pairs as well.
\end{remark}

%
%
%

%
%

\bibliographystyle{alpha}
\bibliography{bibs}

\newcommand{\etalchar}[1]{$^{#1}$}
\def\cprime{$'$} \def\cprime{$'$}
\begin{thebibliography}{BHPVdV04}

\bibitem[AO13]{MR3096524}
Valery Alexeev and Dmitri Orlov.
\newblock Derived categories of {B}urniat surfaces and exceptional collections.
\newblock {\em Math. Ann.}, 357(2):743--759, 2013.

\bibitem[BGvBKS15]{MR3361723}
Christian B{\"o}hning, Hans-Christian Graf~von Bothmer, Ludmil Katzarkov, and
  Pawel Sosna.
\newblock Determinantal {B}arlow surfaces and phantom categories.
\newblock {\em J. Eur. Math. Soc. (JEMS)}, 17(7):1569--1592, 2015.

\bibitem[BGvBS14]{MR3177299}
Christian B{{\"o}}hning, Hans-Christian Graf~von Bothmer, and Pawel Sosna.
\newblock On the {J}ordan-{H}{\"o}lder property for geometric derived
  categories.
\newblock {\em Adv. Math.}, 256:479--492, 2014.

\bibitem[BHPVdV04]{Barth-Hulek-Peters-Van_de_Ven}
Wolf~P. Barth, Klaus Hulek, Chris A.~M. Peters, and Antonius Van~de Ven.
\newblock {\em Compact complex surfaces}, volume~4 of {\em Ergebnisse der
  Mathematik und ihrer Grenzgebiete. 3. Folge. A Series of Modern Surveys in
  Mathematics [Results in Mathematics and Related Areas. 3rd Series. A Series
  of Modern Surveys in Mathematics]}.
\newblock Springer-Verlag, Berlin, second edition, 2004.

\bibitem[BK89]{Bondal-Kapranov_Serre}
Alexey Bondal and Mikhail Kapranov.
\newblock Representable functors, {S}erre functors, and reconstructions.
\newblock {\em Izv. Akad. Nauk SSSR Ser. Mat.}, 53(6):1183--1205, 1337, 1989.

\bibitem[Blo75]{MR0371891}
Spencer Bloch.
\newblock {$K_{2}$} of {A}rtinian {$Q$}-algebras, with application to algebraic
  cycles.
\newblock {\em Comm. Algebra}, 3:405--428, 1975.

\bibitem[BM77]{MR0491719}
Enrico Bombieri and David Mumford.
\newblock Enriques' classification of surfaces in char. {$p$}. {II}.
\newblock In {\em Complex analysis and algebraic geometry}, pages 23--42.
  Iwanami Shoten, Tokyo, 1977.

\bibitem[BN93]{MR1214458}
Marcel B{{\"o}}kstedt and Amnon Neeman.
\newblock Homotopy limits in triangulated categories.
\newblock {\em Compositio Math.}, 86(2):209--234, 1993.

\bibitem[BO95]{Bondal-Orlov_semiorthogonal}
Alexey Bondal and Dmitri Orlov.
\newblock Semiorthogonal decomposition for algebraic varieties.
\newblock arXiv:alg-geom/9506012, 1995.

\bibitem[Bon89]{Bondal_RAACS}
Alexey Bondal.
\newblock Representations of associative algebras and coherent sheaves.
\newblock {\em Izv. Akad. Nauk SSSR Ser. Mat.}, 53(1):25--44, 1989.

\bibitem[BvBS13]{MR3062745}
Christian B{{\"o}}hning, Hans-Christian~Graf von Bothmer, and Pawel Sosna.
\newblock On the derived category of the classical {G}odeaux surface.
\newblock {\em Adv. Math.}, 243:203--231, 2013.

\bibitem[BvdB03]{Bondal-van_den_Bergh}
A.~Bondal and M.~van~den Bergh.
\newblock Generators and representability of functors in commutative and
  noncommutative geometry.
\newblock {\em Mosc. Math. J.}, 3(1):1--36, 258, 2003.

\bibitem[Cal00]{caldararu2000derived}
Andrei~Horia Caldararu.
\newblock {\em Derived categories of twisted sheaves on Calabi-Yau manifolds}.
\newblock PhD thesis, Cornell University, 2000.

\bibitem[CD89]{MR986969}
Fran{\c{c}}ois~R. Cossec and Igor~V. Dolgachev.
\newblock {\em Enriques surfaces. {I}}, volume~76 of {\em Progress in
  Mathematics}.
\newblock Birkh\"auser Boston Inc., Boston, MA, 1989.

\bibitem[dJ03]{de2003result}
Aise~Johan de~Jong.
\newblock A result of {G}abber.
\newblock {\em preprint}, 25:36--57, 2003.

\bibitem[FGI{\etalchar{+}}05]{MR2222646}
Barbara Fantechi, Lothar G{\"o}ttsche, Luc Illusie, Steven~L. Kleiman, Nitin
  Nitsure, and Angelo Vistoli.
\newblock {\em Fundamental algebraic geometry}, volume 123 of {\em Mathematical
  Surveys and Monographs}.
\newblock American Mathematical Society, Providence, RI, 2005.
\newblock Grothendieck's FGA explained.

\bibitem[GS13]{MR3077896}
Sergey Galkin and Evgeny Shinder.
\newblock Exceptional collections of line bundles on the {B}eauville surface.
\newblock {\em Adv. Math.}, 244:1033--1050, 2013.

\bibitem[Har77]{Hartshorne}
Robin Hartshorne.
\newblock {\em Algebraic geometry}.
\newblock Springer-Verlag, New York, 1977.
\newblock Graduate Texts in Mathematics, No. 52.

\bibitem[HL10]{Huybrechts-Lehn}
Daniel Huybrechts and Manfred Lehn.
\newblock {\em The geometry of moduli spaces of sheaves}.
\newblock Cambridge Mathematical Library. Cambridge University Press,
  Cambridge, second edition, 2010.

\bibitem[Hor79]{MR517773}
Eiji Horikawa.
\newblock Algebraic surfaces of general type with small {$c^{2}_{1}$}. {IV}.
\newblock {\em Invent. Math.}, 50(2):103--128, 1978/79.

\bibitem[HT15]{hosono2015derived}
Shinobu Hosono and Hiromichi Takagi.
\newblock Derived categories of artin-mumford double solids.
\newblock {\em arXiv preprint arXiv:1506.02744}, 2015.

\bibitem[Huy06]{MR2244106}
D.~Huybrechts.
\newblock {\em Fourier-{M}ukai transforms in algebraic geometry}.
\newblock Oxford Mathematical Monographs. The Clarendon Press Oxford University
  Press, Oxford, 2006.

\bibitem[IK15]{MR3302614}
Colin Ingalls and Alexander Kuznetsov.
\newblock On nodal {E}nriques surfaces and quartic double solids.
\newblock {\em Math. Ann.}, 361(1-2):107--133, 2015.

\bibitem[Kaw04]{Kawamata_EDCSSS}
Yujiro Kawamata.
\newblock Equivalences of derived categories of sheaves on smooth stacks.
\newblock {\em Amer. J. Math.}, 126(5):1057--1083, 2004.

\bibitem[Kaw05]{Kawamata_LCBMDC}
Yujiro Kawamata.
\newblock Log crepant birational maps and derived categories.
\newblock {\em J. Math. Sci. Univ. Tokyo}, 12(2):211--231, 2005.

\bibitem[Kaw09]{MR2483950}
Yujiro Kawamata.
\newblock Derived categories and birational geometry.
\newblock In {\em Algebraic geometry---{S}eattle 2005. {P}art 2}, volume~80 of
  {\em Proc. Sympos. Pure Math.}, pages 655--665. Amer. Math. Soc., Providence,
  RI, 2009.

\bibitem[KM97]{MR1432041}
Se{\'a}n Keel and Shigefumi Mori.
\newblock Quotients by groupoids.
\newblock {\em Ann. of Math. (2)}, 145(1):193--213, 1997.

\bibitem[Knu71]{MR0302647}
Donald Knutson.
\newblock {\em Algebraic spaces}.
\newblock Lecture Notes in Mathematics, Vol. 203. Springer-Verlag, Berlin,
  1971.

\bibitem[Kol86]{MR825838}
J{\'a}nos Koll{\'a}r.
\newblock Higher direct images of dualizing sheaves. {I}.
\newblock {\em Ann. of Math. (2)}, 123(1):11--42, 1986.

\bibitem[Kuz06]{MR2238172}
Alexander Kuznetsov.
\newblock Hyperplane sections and derived categories.
\newblock {\em Izv. Ross. Akad. Nauk Ser. Mat.}, 70(3):23--128, 2006.

\bibitem[Kuz11]{MR2801403}
Alexander Kuznetsov.
\newblock Base change for semiorthogonal decompositions.
\newblock {\em Compos. Math.}, 147(3):852--876, 2011.

\bibitem[Kuz13]{kuznetsov2013simple}
Alexander Kuznetsov.
\newblock A simple counterexample to the {J}ordan-{H}\"older property for
  derived categories.
\newblock {\em arXiv preprint arXiv:1304.0903}, 2013.

\bibitem[LMB00]{MR1771927}
G{\'e}rard Laumon and Laurent Moret-Bailly.
\newblock {\em Champs alg\'ebriques}, volume~39 of {\em Ergebnisse der
  Mathematik und ihrer Grenzgebiete. 3. Folge. A Series of Modern Surveys in
  Mathematics [Results in Mathematics and Related Areas. 3rd Series. A Series
  of Modern Surveys in Mathematics]}.
\newblock Springer-Verlag, Berlin, 2000.

\bibitem[Mum68]{MR0249428}
David Mumford.
\newblock Rational equivalence of {$0$}-cycles on surfaces.
\newblock {\em J. Math. Kyoto Univ.}, 9:195--204, 1968.

\bibitem[Nav10]{navas2010fourier}
Hermes Jackson~Martinez Navas.
\newblock {\em Fourier-Mukai transform for twisted sheaves}.
\newblock PhD thesis, Universit{\"a}ts-und Landesbibliothek Bonn, 2010.

\bibitem[Oka11]{MR2838062}
Shinnosuke Okawa.
\newblock Semi-orthogonal decomposability of the derived category of a curve.
\newblock {\em Adv. Math.}, 228(5):2869--2873, 2011.

\bibitem[Orl92]{Orlov_PB}
Dmitri Orlov.
\newblock Projective bundles, monoidal transformations, and derived categories
  of coherent sheaves.
\newblock {\em Izv. Ross. Akad. Nauk Ser. Mat.}, 56(4):852--862, 1992.

\bibitem[Ros09]{rosay2009some}
Fabrice Rosay.
\newblock Some remarks on the group of derived autoequivalences.
\newblock {\em arXiv preprint arXiv:0907.3880}, 2009.

\bibitem[Rou08]{MR2434186}
Rapha{\"e}l Rouquier.
\newblock Dimensions of triangulated categories.
\newblock {\em J. K-Theory}, 1(2):193--256, 2008.

\bibitem[{Sta}17]{stacks-project}
The {Stacks Project Authors}.
\newblock exit{Stacks Project}.
\newblock \url{http://stacks.math.columbia.edu}, 2017.

\bibitem[Tot04]{MR2108211}
Burt Totaro.
\newblock The resolution property for schemes and stacks.
\newblock {\em J. Reine Angew. Math.}, 577:1--22, 2004.

\bibitem[Uen73]{MR0360582}
Kenji Ueno.
\newblock Classification of algebraic varieties. {I}.
\newblock {\em Compositio Math.}, 27:277--342, 1973.

\bibitem[Zuc03]{MR1980618}
Francesco Zucconi.
\newblock Surfaces with {$p_g=q=2$} and an irrational pencil.
\newblock {\em Canad. J. Math.}, 55(3):649--672, 2003.

\end{thebibliography}

\end{document}